\documentclass[10pt,a4paper,twoside]{article}

\usepackage[latin1]{inputenc}
\usepackage{amssymb}
\usepackage{amsmath}
\usepackage{amsfonts}
\usepackage{amsbsy}
\usepackage{eucal}
\usepackage{graphicx}
\usepackage{multirow}
\usepackage[english]{babel}
\usepackage{mathptmx}
\usepackage{float}
\usepackage[font=small]{caption}
\usepackage{array}
\usepackage[colorlinks,citecolor=blue,urlcolor=blue]{hyperref}
\usepackage{amsthm}
\usepackage{comment}
\usepackage{setspace}
\usepackage{bm}
\usepackage{arxiv}
\allowdisplaybreaks
\newtheorem{theorem}{Theorem}[section]

\newtheorem{usl}{Condition}
\begin{document}
\title{Asymptotic distribution of certain degenerate V- and U-statistics with estimated parameters}

\author{Marija Cupari\'c\thanks{marijar@matf.bg.ac.rs} ,  Bojana Milo\v sevi\'c\thanks{bojana@matf.bg.ac.rs} , Marko Obradovi\' c \thanks{marcone@matf.bg.ac.rs}
         \\\medskip
{\small Faculty of Mathematics, University of Belgrade, Studenski trg 16, Belgrade, Serbia}}

\date{}

\maketitle

\begin{abstract}
The asymptotic distribution of a wide class of V- and U-statistics with estimated parameters is derived in the case when the kernel is not necessarily differentiable along the parameter. The results have their application in  goodness-of-fit problems. 
\end{abstract}
%%%%% end %%%%%%%%%%%

%%%%% Keywords %%%%%%%%%%%
{\small \textbf{ keywords:} degenerate V-statistics, characterization, goodness-of-fit}
% 

% %%%% AMS subject classifications %%%%
\textbf{MSC(2010):} 62E20,  62G20

%%%% maketitle %%%%%
\maketitle

%%%% Start %%%%%%
\section{Introduction}

V- and U-statistics frequently appear in inferential procedures as estimators and test statistics. In many cases they depend on nuisance parameters, and, in order to keep the
inference as broad as possible, estimators of such parameters are included in the statistic.

Although some earlier papers (see e.g. \cite{sukhatme1958testing,raghavachari1965two,gupta1967asymptotically}) consider some particular cases, 
the first general study devoted to U-statistics with estimated parameters was done in \cite{randles1982asymptotic}, where the effect 
of estimating parameters in asymptotically normal (non-degenerate) U-statistics was examined. De Wet and Randles 
\cite{de1987effect} studied a special case of degenerate V-statistics of order 2 obtained as an integrated square of a 
V-statistic of order 1. They proved that, under some regularity conditions, the limiting distribution 
is certain infinite linear combination of $\chi^2_1$ random variables. Their method was subsequently used for examining particular goodness-of-fit tests in several papers (see e.g. \cite{jimenez2003bootstrapping,wei2016weighted,henze2020test}).

%{\color{purple} ovaj pasus se brise}
%Creating tests based on characterizations has recently become a popular approach in goodness-of-fit and symmetry testing. Their main asset is that they are often distribution-free under the null hypothesis. Characterization based tests for the exponential, normal, Pareto, logistic and other families of distributions
%can be found in e.g. \cite{},
%and an up-to-date review is available in \cite{nikitin2017tests}. 

Here we extend the result from \cite{de1987effect} by considering the case when the V-statistic of order 1
is replaced with a V-statistic of an arbitrary order. The practical value of this result is reflected in the fact that 
a class of characterization based test statistics are of this form. This class includes $L^2$-type tests employing V-empirical Laplace transforms of equidistributed random variables (see  \cite{cuparic2018new,revista}), but it is not limited to them. In fact, instead of  V-empirical Laplace transforms one could employ  V-empirical distribution functions (as e.g. in \cite{IMJournal}), characteristic functions, densities, etc. 
Exploring the asymptotics of such tests is closely related to V- and U-empirical functions with estimated parameters. Additionally, it enables calculation of the approximate Bahadur efficiency, one of important tools for test comparison (see e.g. \cite{cuparic}). The exceptional performance of the tests from \cite{cuparic2018new} set the motivation to generalize the result and hence facilitate further research on this topic. %{  In particular, results would allowed further development of characterization based approach in goodness-of-testing, dirrection that become very popular in last decade CITAT. Due to lack of the results from this paper, such tests based on quadratic difference among two V-emprical functions with estimated parameters, were not well explored....}

In addition, we present an analogous result for U-statistics. This result, using a different method of proof and slightly more general form, is also obtained in \cite[Theorem 3]{arcones2007two}.

%\section{Preliminaries    ili da ovo udje u introduction?}

Let $X_1,\ldots, X_n$ be a sequence of i.i.d. random variables with distribution function F  that depends on an unknown parameter $\lambda$. A V-statistic of order $m$
with estimated parameter, where $\widehat{\lambda}_n$ is a consistent estimator of $\lambda$, is %given by
\begin{align}\label{Vocenjeno}
     V_{n}(\widehat{\lambda}_n)=\frac{1}{n^{m}}\sum_{i_1,...,i_{m}=1}^n\Phi(X_{i_1},...,X_{i_{m}};\widehat{\lambda}_n).
\end{align}
If its kernel $\Phi(x_1,...,x_m;\lambda)$ satisfies that
%\begin{equation*}\label{1uslovDeg}
%    \varphi_1(x)=E_{\lambda}(\Phi(x,X_2,...,X_m;\lambda))=0
%\end{equation*}
%for all $x$, and 
$\varphi_1(x;\lambda)=0,$ for all $x$, and 
   $E_{\lambda}\varphi^2_2(X_1,X_2;\lambda)>0$,
where $\varphi_1(x)=E_{\lambda}(\Phi(x,X_2,...,X_m;\lambda))$ and $\varphi_2(x,y;\lambda)=E_{\lambda}(\Phi(x,y,X_3,...,X_m;\lambda))$ are the first and second projection of the kernel $\Phi$,
%\begin{align}\label{phi2}
%    \varphi_2(x,y;\lambda)=E_{\lambda}(\Phi(x,y,X_3,...,X_m;\lambda)),
%\end{align}
then \eqref{Vocenjeno} is a V-statistic with a weakly degenerate kernel.

If the parameter $\lambda$ is known, then the asymptotic distribution is (see e.g. \cite{serfling2009approximation})
\begin{align}\label{raspodelaPoznato}
 nV_n(\lambda)\overset{D}{\to}\binom{2m}{2}\sum_{k=1}^{+\infty}\upsilon_kZ_{k}^2,
\end{align}
where $\{\upsilon_k\}, k=1,2,...,$ is the sequence of eigenvalues of the operator $A$ defined on $L^2(\mathbb{R},F)$ as
\begin{equation}\label{operatorA}
    Aq(x)=\int_{R}\varphi_2(x,y)q(y)dF(y),
\end{equation}
and $Z_{k},k=1,2,...,$ are i.i.d random variables with standard normal distribution.

In \cite{de1987effect} the following special case was considered
\begin{align*}
 V_n(\widehat{\lambda})=\int_{-\infty}^\infty \Big(\frac1n \sum_{i=1}^n g(X_i,t;\widehat{\lambda}_{  n})\Big)^2 dM(t),
\end{align*}
where $g$ is a function satisfying some regularity conditions and $M(t)$ is a finite measure. The kernel of 
$V_n(\widehat{\lambda})$ is then
\begin{align*}
 \Phi(x_1,x_2,\widehat{\lambda})=\int_{-\infty}^\infty g(x_1,t;\widehat{\lambda}_{  n})g(x_2,t;\widehat{\lambda}_{  n})dM(t).
\end{align*}

\section{Main results}

Generalizing the idea from \cite{de1987effect}, we consider the  V-statistics of order $2m$ with the following 
symmetrized kernel
\begin{align}\label{jezgro}
        \Phi(x_{1},...,x_{{2m}};\widehat{\lambda}_n)&=\frac{1}{(2m)!}\sum\limits_{\pi\in\Pi(2m)}\int\limits_{-\infty}^{+\infty} 
        g(x_{\pi(1)},...,x_{\pi(m)},t;\widehat{\lambda}_n) g(x_{\pi({m+1})},...,x_{\pi({2m})},t;\widehat{\lambda}_n)dM(t),
\end{align}
where $\Pi(2m)$ is the set of all permutations of $\{1,2,...,2m\}$ and $M(t)$ is a finite measure. 
Without loss of generality, we assume that the function $g$ is a symmetric function of its first $m$ arguments.

Next we state some regularity conditions imposed on the smoothness of the  function $g$ and the rate of convergence  of the estimator $\widehat{\lambda}_n$.

\begin{usl}\label{uslov1}
Suppose that $\mu(t;\gamma)=E_\lambda(g(X_1,...,X_m,t;\gamma))$ exists for all $\gamma$ in the  neighbourhood of $\lambda$
and it satisfies $\mu(t;\lambda)\equiv 0$ for each $t$. 
Suppose additionally that for all $\varepsilon>0$ there exists a  ball $S$ with finite radius in $\mathbb{R}^p$ centered at $\lambda$, such that
for $\gamma\in S$ it holds
\begin{equation*}
    \frac{1}{||\gamma-\lambda||^2}\int\limits_{-\infty}^{+\infty}(\mu(t;\gamma)-d_1\mu(t;\lambda)^T(\gamma-\lambda))^2dM(t)<\varepsilon,
\end{equation*}
where $d_1\mu(t;\!\lambda)$ is the vector of partial derivatives %{   Mislim da ovo nije Gatoov izvod, u pitanju je realna funkcija vise promenljivih}
of $\mu(t;\gamma)$ at $\gamma=\lambda$, 
satisfying  
\begin{equation*}
    \int\limits_{-\infty}^{+\infty}(d_1\mu(t;\lambda)_r)^2dM(t)<\infty,
\end{equation*}
for $r=1,...,p$ where $d_1\mu(\cdot)_r$ is the $r-$th component of the vector $d_1\mu(\cdot)$.
\end{usl}
  
\begin{usl}\label{uslov2}
  Suppose  
  \begin{align}
   \widehat{\lambda}_n=\lambda+\frac{1}{n}\displaystyle\sum_{i=1}^n\alpha(X_i)+o_p\Big(\frac{1}{\sqrt{n}}\Big),\; n\to\infty,
  \end{align}
  where $E(\alpha(X_i)_r)=0$ and $E(\alpha(X_i)_r\alpha(X_i)_{r'})<\infty$ for all $1\leq r\leq r'\leq p$.
\end{usl} 

\begin{usl}\label{uslov3}
  For any variation $(i_1,...,i_m)$ of indices from $\{1,2,....,n\}$, there exists a neighbourhood $K(\lambda)$ of $\lambda$ and a constant $C>0$, such that
  for $\gamma\in K(\lambda)$ and for any ball $D(\gamma,d)$ centered at $\gamma$ with radius $d$ such that $D(\gamma,d)\subset K(\lambda)$, it holds
    \begin{align*}
         \int\limits_{-\infty}^{+\infty}(E(\sup_{\gamma'\in D(\gamma, d)}|g(X_{i_1},...,X_{i_m},t;\gamma')-g(X_{i_1},...,X_{i_m},t;\gamma)|))^2dM(t)\leq Cd^2.
      \end{align*}
     
    In addition, for every $\varepsilon>0$ there exists $d^*>0$ such that for $0<d<d^*$ holds
      \begin{align*}
         % \int\limits_{-\infty}^{+\infty}E(\sup_{\gamma'\in D(\gamma, d)}|g(X_{i_1},...,X_{i_m},t;\gamma')-g(X_{i_1},...,X_{i_m},t;\gamma)|^4)dM(t)<\epsilon.
         \int\limits_{-\infty}^{+\infty}E(\sup_{\gamma'\in D(\gamma, d)}|g(X_{i_1},...,X_{i_m},t;\gamma')-g(X_{i_1},...,X_{i_m},t;\gamma)|)^4dM(t)<\varepsilon.
      \end{align*}
     %  \begin{equation*}\color{violet}
         % \int\limits_{-\infty}^{+\infty}E(\sup_{\gamma'\in D(\gamma, d)}|g(X_{i_1},...,X_{i_m},t;\gamma')-g(X_{i_1},...,X_{i_m},t;\gamma)|^4)dM(t)<\epsilon.
     %    \int\limits_{-\infty}^{+\infty}E(\sup_{\gamma'\in D(\gamma, d)}|g(X_{i_1},...,X_{i_m},t;\gamma')-g(X_{i_1},...,X_{i_m},t;\gamma)|^4)dM(t)<\epsilon.
     % \end{equation*}
\end{usl}

Following \cite{de1987effect} we define a new statistic, that depends on the true value of the parameter
\begin{equation}\label{Vpom}
\begin{split}
    V^*_n(\lambda)&=\frac{1}{n^{2m}}\sum_{i_1,...,i_{2m}}\Phi_*(X_{i_1},...,X_{i_{2m}};\lambda)\\&=\frac{1}{n^{2m}}\sum_{i_1,...,i_{2m}}\int\limits_{-\infty}^{+\infty}\bigg(g(X_{i_1},...,X_{i_m},t;\lambda)+d_1\mu(t;\lambda)^T\frac{1}{m}\sum_{i_j\in\{i_1,...,i_m\}}\alpha(X_{i_j})\bigg)%+d_1\mu(t;\lambda)'\frac{1}{n}\sum_{j=1}^n\alpha(x_j)\bigg)
   \\&\times\bigg(g(X_{i_{m+1}},...,X_{i_{2m}},t;\lambda)+d_1\mu(t;\lambda)^T\frac{1}{m}\sum_{i_k\in\{i_{m+1},..,i_{2m}\}}\alpha(X_{i_k})\bigg)dM(t),
    \end{split}
\end{equation}
where the function $\alpha(\cdot)$ is defined in Condition \ref{uslov2}. Let $\varphi^*_1(x;\lambda)$ and $\varphi^*_2(x,y;\lambda)$ be the first and second projection of the symmetrized version of the kernel $\Phi_*$. Using Conditions 1 and 2 it can be easily shown that $\varphi^*_1(x;\lambda)=0.$
    
%Notice that
%\begin{align*}
%    E\Phi^*(X_1,...,X_{2m};\lambda)&=E\bigg(\int\limits_{-\infty}^{+\infty}\bigg(g(X_{1},...,X_{m},t;\lambda)+d_1\mu(t;\lambda)^T\frac{1}{m}\sum_{j\in\{1,...,m\}}\alpha(X_{j})\bigg)  \\&\times\bigg(g(X_{{m+1}},...,X_{{2m}},t;\lambda)+d_1\mu(t;\lambda)^T\frac{1}{m}\sum_{k\in\{{m+1},..,{2m}\}}\alpha(X_{k})\bigg)dM(t)\bigg)\\&=\int\limits_{-\infty}^{+\infty}E\bigg(g(X_{1},...,X_{m},t;\lambda)+d_1\mu(t;\lambda)^T\frac{1}{m}\sum_{j\in\{1,...,m\}}\alpha(X_{j})\bigg)
  % \\&\times E\bigg(g(X_{{m+1}},...,X_{{2m}},t;\lambda)+d_1\mu(t;\lambda)^T\frac{1}{m}\sum_{k\in\{{m+1},..,{2m}\}}\alpha(X_{k})\bigg)dM(t)\\&=0
%\end{align*}
%shto sledi iz uslova \ref{uslov1} i \ref{uslov2}. Takodje,
%\begin{align*}
%    \varphi^*_1&(x;\lambda)=
%    E(\Phi_*(x,X_2,...,X_{2m};\lambda))%\\&=
    %E\bigg(\int\limits_{-\infty}^{+\infty}\bigg(g(x,X_{2},...,X_{m},t;\lambda)+d_1\mu(t;\lambda)^T\frac{1}{m}\sum_{j\in\{1,...,m\}}\alpha(X_{j})\bigg)
   %\\&\color{black}\times\bigg(g(X_{{m+1}},...,X_{{2m}},t;\lambda)+d_1\mu(t;\lambda)^T\frac{1}{m}\sum_{k\in\{{m+1},..,{2m}\}}\alpha(X_{k})\bigg)dM(t)\bigg)\\&\color{black}=
 %  \int\limits_{-\infty}^{+\infty}E\bigg(g(x,X_{2},...,X_{m},t;\lambda)+d_1\mu(t;\lambda)^T\frac{1}{m}(\alpha(x)+\sum_{j\in\{2,...,m\}}\alpha(X_{j}))\bigg)
%   \\&\times E\bigg(g(X_{{m+1}},...,X_{{2m}},t;\lambda)+d_1\mu(t;\lambda)^T\frac{1}{m}\sum_{k\in\{{m+1},..,{2m}\}}\alpha(X_{k})\bigg)dM(t)\bigg)
%=0
%\end{align*}
%which follows from Conditions \ref{uslov1} and \ref{uslov2}. 
Moreover,
\begin{align}\label{fi2}
    &\varphi^*_2(x_1,x_{m+1};\lambda)=\frac{2m^2(2m-2)!}{(2m)!} E(\Phi_*(x_1,X_2,...,X_{m},x_{m+1},X_{m+2}...,X_{2m};\lambda))\nonumber\\&=
    % \frac{2m^2(2m-2)!}{(2m)!}\color{black}E\bigg(\int\limits_{-\infty}^{+\infty}\bigg(g(x_1,X_{2},...,X_{m},t;\lambda)+d_1\mu(t;\lambda)^T\frac{1}{m}\Big(\alpha(x_1)+\sum_{j\in\{2,...,m\}}\alpha(X_{j})\Big)\bigg)\nonumber
   %\\&\color{black}\hspace{1cm}\times\bigg(g(x_{m+1},X_{{m+2}},...,X_{{2m}},t;\lambda)+d_1\mu(t;\lambda)^T\frac{1}{m}\Big(\alpha(x_{m+1})+\sum_{k\in\{{m+2},..,{2m}\}}\alpha(X_{k})\Big)\bigg)dM(t)\bigg)\nonumber\\&\color{black}= \frac{2m^2(2m-2)!}{(2m)!}\color{black}\int\limits_{-\infty}^{+\infty}E\bigg(g(x_1,X_{2},...,X_{m},t;\lambda)+d_1\mu(t;\lambda)^T\frac{1}{m}\Big(\alpha(x_1)+\sum_{j\in\{2,...,m\}}\alpha(X_{j})\Big)\bigg)\nonumber
   %\\&\color{black}\hspace{1cm}\times E\bigg(g(x_{m+1},X_{{m+2}},...,X_{{2m}},t;\lambda)+d_1\mu(t;\lambda)^T\frac{1}{m}\Big(\alpha(x_{m+1})+\sum_{k\in\{{m+2},..,{2m}\}}\alpha(X_{k})\Big)\bigg)dM(t)\nonumber\\&= 
   \frac{2m^2(2m-2)!}{(2m)!}\int\limits_{-\infty}^{+\infty}\bigg(g_1(x_1,t;\lambda)+d_1\mu(t;\lambda)^T\frac{1}{m}\alpha(x_{1})\bigg)
   \bigg(g_1(x_{m+1},t;\lambda)+d_1\mu(t;\lambda)^T\frac{1}{m}\alpha(x_{m+1})\bigg)dM(t),
\end{align}
where $g_1(x,t;\lambda)=E_\lambda(g(x,X_{2},...,X_{m},t;{\lambda}))$.
Since $V^*_n(\lambda)$ is a weakly degenerate V-statistic, it holds 
\begin{equation}\label{raspodelaV*}
    nV^*_n(\lambda)\stackrel{d}{\rightarrow}\binom{2m}{2}\sum_{k=1}^{+\infty}\upsilon^*_kZ_{k}^2,
\end{equation}
where $\{\upsilon^*_k\}$ is the sequence of eigenvalues of  integral operator $A^*$ defined with
\begin{equation}\label{operatorA*}
    A^*q(x)=\int_{R}\varphi_2^*(x,y;\lambda)q(y)dF(y),
\end{equation}
and $\{Z_k\}$ is the sequence of i.i.d. random variables with standard normal distribution.
The following theorem gives the asymptotic distribution of a V-statistic with kernel \eqref{jezgro}.
\begin{theorem}\label{raspodela} 
%Neka je $X_1,...,X_n$ prost sluchajan uzorak sa funkcijom raspo\-de\-le $F$. 
Suppose Conditions \ref{uslov1}-\ref{uslov3} are satisfied and, additionally, that it holds $E(\Phi^2_*(\!X_1,\!...,X_{2m}; \lambda))\!<\!+\!\infty$ and 
$E(\Phi_*(X_{i_1},...,X_{i_{2m}}; \lambda))<+\infty,$ for all variations $i_1,\ldots,i_m$, $1\leq i_1,...,i_{2m}\leq n$. 
%$E(\int_{-\infty}^{+\infty}g(X_1,...,X_m,t;\widehat{\lambda}_n)dM(t))<\infty$. 
Then 
\begin{align}
 n(V_n(\widehat{\lambda}_n)-V^*_n(\lambda))\stackrel{P}{\rightarrow}0
\end{align}
 Moreover, if $E(\varphi_2^*(X_1,X_2;\lambda))^2>0$, then
\begin{equation*}
    nV_n(\widehat{\lambda}_n)\stackrel{D}{\rightarrow}\binom{2m}{2}\sum_{k=1}^{+\infty}\upsilon_k^*Z_{k}^2,
\end{equation*}
where $\{\upsilon_k^*\}$ is the sequence of eigenvalues of the operator $A^*$ defined in \eqref{operatorA*}
%where $\varphi_2^*(x,y)$ is the second projection of the symmetrized version of the kernel $\Phi_*(x_1,...,x_{2m})$, 
and
$\{Z_{k}\}$ %,k=1,2,...,$ 
is the sequence of i.i.d. random variables with standard normal distribution.
\end{theorem}
\begin{proof}
%{\bf Proof:} 
The idea is to show the equidistribution of $nV_n(\widehat{\lambda}_{n})$ and $nV^*_n(\lambda)$, from where, taking into account \eqref{raspodelaV*}, will follow the statement of the theorem.

Without loss of generality we assume  $\int_{-\infty}^\infty dM(t)=1$.
Define an auxiliary statistic $Y_n$ as 
\begin{equation*}
    Y_n=\int\limits_{-\infty}^{+\infty}\bigg(\frac{1}{n^{m}}\sum_{i_1,...,i_m}g(X_{i_1},...,X_{i_m},t;\lambda)+\mu(t;\widehat{\lambda}_n)\bigg)^2dM(t).
\end{equation*}
It is sort of midway from $V^*_n(\lambda)$ to $V_n(\widehat{\lambda}_n)$, depending both on the true value $\lambda$ and its estimator $\widehat{\lambda}_n$.
We show that $nY_n-nV^*_n(\lambda)\stackrel{P}{\to}0$ and $nV_n(\widehat{\lambda}_n)-nY_n\stackrel{P}{\to}0$. 

Consider first the difference between $Y_n$  and $V^*_n(\lambda)$. Then, using the identity $a^2-b^2=(a-b)^2+2b(a-b)$ we get
\begin{align*}
   %\begin{split}
       &Y_n-V^*_n(\lambda)%=\int\limits_{-\infty}^{+\infty}\bigg(\frac{1}{n^{m}}\sum_{i_1,...,i_m}g(X_{i_1},...,X_{i_m},t;\lambda)+\mu(t;\widehat{\lambda}_n)\bigg)^2dM(t)\\&-\!\int\limits_{-\infty}^{+\infty}\!\bigg(\!\frac{1}{n^{m}}\!\sum_{i_1,...,i_m}\!\Big(\!g(X_{i_1},...,X_{i_m},t;\!\lambda)\!+\!d_1\mu(t;\!\lambda)^T\frac{1}{m}\!\sum\limits_{i_k\in\{i_1,...,i_m\}}\!\!\alpha(X_{i_k})\!\Big)\!\bigg)^2\!dM(t)    \\&
       =       
       \int\limits_{-\infty}^{+\infty}\bigg(\mu(t;\widehat{\lambda}_n)-d_1\mu(t;\lambda)^T\frac{1}{n}\sum_{k=1}^n\alpha(X_{k})\bigg)^2dM(t)\\&+\!2\int\limits_{-\infty}^{+\infty}\!\bigg(\frac{1}{n^{m}}\sum_{i_1,...,i_m}\bigg(g(X_{i_1},...,X_{i_m},t;\lambda)+d_1\mu(t;\lambda)^T\frac{1}{m}\sum_{i_k\in\{i_1,...,i_m\}}\alpha(X_{i_k})\bigg)\!\bigg)\!\bigg(\!\mu(t;\widehat{\lambda}_n)\!-\!d_1\mu(t;\lambda)^T\frac{1}{n}\sum_{k=1}^n\alpha(X_{k})\bigg)dM(t).
    %\end{split}
\end{align*}

Therefore, using the Cauchy-Schwarz inequality, we obtain
\begin{equation}\label{Y-V}
   \begin{split}
   n(Y_n-V^*_n(\lambda))&\leq n\int\limits_{-\infty}^{+\infty}\bigg(\mu(t;\widehat{\lambda}_n)-d_1\mu(t;\lambda)^T\frac{1}{n}\sum_{k=1}^n\alpha(X_{k})\bigg)^2dM(t)
  \\&+2\bigg(\!\int\limits_{-\infty}^{+\infty}n\bigg(\!\frac{1}{n^{m}}\sum_{i_1,...,i_m}g(X_{i_1},...,X_{i_m},t;\lambda)\!+\!d_1\mu(t;\lambda)^T\frac{1}{n}\!\sum_{k=1}^n\!\alpha(X_{k})\!\bigg)^2\!dM(t)\!\bigg)^\frac{1}{2}
  \\&\times\bigg(\int\limits_{-\infty}^{+\infty}n\bigg(\mu(t;\widehat{\lambda}_n)-d_1\mu(t;\lambda)^T\frac{1}{n}\sum_{k=1}^n\alpha(X_{k})\bigg)^2dM(t)\bigg)^\frac{1}{2}
      % n(Y_n-V_n(\lambda))&\leq        n\int\limits_{-\infty}^{+\infty}\bigg(\mu(t;\widehat{\lambda}_n)-d_1\mu(t;\lambda)(\widehat{\lambda}_n-\lambda)\bigg)^2dM(t)\\&+2n\bigg(\int\limits_{-\infty}^{+\infty}\bigg(\frac{1}{n^{m}}\sum_{i_1,...,i_m}g(X_{i_1},...,X_{i_m},t;\lambda)+d_1\mu(t;\lambda)(\widehat{\lambda}_n-\lambda)\bigg)^2dM(t)\bigg)^\frac{1}{2}\\&\times\bigg(\int\limits_{-\infty}^{+\infty}\bigg(\mu(t;\widehat{\lambda}_n)-d_1\mu(t;\lambda)(\widehat{\lambda}_n-\lambda)\bigg)^2dM(t)\bigg)^\frac{1}{2}.
   \\&= n\int\limits_{-\infty}^{+\infty}\bigg(\mu(t;\widehat{\lambda}_n)-d_1\mu(t;\lambda)^T\frac{1}{n}\sum_{k=1}^n\alpha(X_{k})\bigg)^2dM(t)
   \\&+2\sqrt{nV^{*}_n(\lambda)}\bigg(\int\limits_{-\infty}^{+\infty}n\bigg(\mu(t;\widehat{\lambda}_n)-d_1\mu(t;\lambda)^T\frac{1}{n}\sum_{k=1}^n\alpha(X_{k})\bigg)^2dM(t)\bigg)^\frac{1}{2}.
   \end{split}
\end{equation}

From the Condition \ref{uslov1} it follows that
\begin{equation}\label{prviClan}
   \begin{split}
       n\int\limits_{-\infty}^{+\infty}\bigg(\mu(t;\widehat{\lambda}_n)-d_1\mu(t;\lambda)^T(\widehat{\lambda}_n-\lambda)\bigg)^2dM(t)<\varepsilon||\sqrt{n}(\widehat{\lambda}_n-\lambda)||^2.
    \end{split}
\end{equation}
Since, due to Condition \ref{uslov2}, $\sqrt{n}(\widehat{\lambda}_n-\lambda)$ is bounded in probability, the first summand of \eqref{Y-V} tends to zero in probability.
{The second summand also tends to zero due to the Slutsky theorem and the fact that $nV^{*}_n(\lambda)$ for V-statistics with kernels satisfying the conditions
of the theorem are bounded in probability.}

Consider now the difference between $V_n(\widehat{\lambda}_n)$ i $Y_n$. Analogously to the previous case we get
\begin{align*}
        n(V_n(\widehat{\lambda}_n)-Y_n)\!&\leq\!
        \int\limits_{-\infty}^{+\infty}n\bigg(\frac{1}{n^{m}}\!\sum_{i_1,...,i_m}\!\Big(g(X_{i_1},...,X_{i_m},t;\widehat{\lambda}_n)\!-\!g(X_{i_1},...,X_{i_m},t;\lambda)\!\Big)-\mu(t;\widehat{\lambda}_n)\bigg)^2dM(t)\\&
        +2\sqrt{nY_n}\bigg(\int\limits_{-\infty}^{+\infty}n\bigg(\frac{1}{n^{m}}\sum_{i_1,...,i_m}\Big(g(X_{i_1},...,X_{i_m},t;\widehat{\lambda}_n)-g(X_{i_1},...,X_{i_m},t;\lambda\Big)-\mu(t;\widehat{\lambda}_n)\bigg)^2dM(t)\bigg)^\frac12.
\end{align*}

Since $nY_n$ is bounded in probability, it suffices to prove that 
\begin{align*}
\int\limits_{-\infty}^{+\infty}\!n\bigg(\!\frac{1}{n^{m}}\!\sum_{i_1,...,i_m}\!\Big(g(X_{i_1},...,X_{i_m},t;\widehat{\lambda}_n)\!-\!g(X_{i_1},...,X_{i_m},t;\lambda{  )}\Big)\!-\!\mu(t;\widehat{\lambda}_n)\!\bigg)^2\!dM(t)
\end{align*}
converges to zero in probability.
Define 
\begin{equation*}
\begin{split}
    Q_n(s,{r})\!=\!\!\int\limits_{-\infty}^{+\infty}\!\bigg(\!\frac{1}{n^{m}}\!\sum_{i_1,...,i_m}\!\Big(\!g\Big(X_{i_1},...,X_{i_m},t;\lambda\!+\!\frac{s}{\sqrt{n}}\!\Big)\!-\!g\!\Big(X_{i_1},...,X_{i_m},t;{\lambda\!+\!\frac{r}{\sqrt{n}}}\!\Big)\!-\!\mu\!\Big(\!t;\lambda\!+\!\frac{s}{\sqrt{n}}\Big)\!+\! \mu\Big(t;\lambda\!+\!\frac{r}{\sqrt{n}}\!\Big)\!\Big)\!\bigg)^2\!dM(t).
\end{split}
\end{equation*}
We need to show that $nQ_n(\sqrt{n}(\widehat{\lambda}_n-\lambda),{  0})\stackrel{P}{\rightarrow}0.$
%\begin{equation*}
%    nQ_n(\sqrt{n}(\widehat{\lambda}_n-\lambda))\stackrel{P}{\rightarrow}0.
%\end{equation*}
Condition \ref{uslov2} ensures the existence of a  ball $S$ in $\mathbb{R}^p$ such that
\begin{equation*}
    P\{\sqrt{n}(\widehat{\lambda}_n-\lambda)\notin S\}\to 0,\;\text{as } n\to \infty.
\end{equation*}
%and for any  $\eta>0$ we have
Then for $\widetilde{\eta}>0$ it holds
\begin{equation*}
\begin{split}
    P\{nQ_n(\sqrt{n}(\widehat{\lambda}_n-\lambda),{  0})\geq\widetilde{\eta}\}%&=P\{nQ_n(\sqrt{n}(\widehat{\lambda}_n-\lambda))>\varepsilon',\sqrt{n}(\widehat{\lambda}_n-\lambda)\in S\}\\&+P\{nQ_n(\sqrt{n}(\widehat{\lambda}_n-\lambda))>\varepsilon',\sqrt{n}(\widehat{\lambda}_n-\lambda)\notin S\}\\&
    \leq P\{\sup\limits_{s\in S}nQ_n(s,{  0})\geq \widetilde{\eta}\}+P\{\sqrt{n}(\widehat{\lambda}_n-\lambda)\notin S\}.
\end{split}
\end{equation*} 
Hence it is enough to prove that $\sup\limits_{s\in S}nQ_n(s,{  0})\stackrel{P}{\rightarrow}0$. 
Following \cite{sukhatme1958testing} (see also \cite{iverson1989effects}), let $\delta>0$ and let $\{S_i\}$, $i=1,\ldots,i_{\max}$, be a collection of balls, centered at $r_i$ with radius $\delta$,
such that for all $s\in S$  there exists at least one $S_i$ such that $s\in S_i$.  The existence of $\{S_i\}$ is ensured 
by the finiteness of $S$.

Then we have
\begin{align*}
 P\{\sup_{s\in S} nQ_n(s,0)\geq \widetilde{\eta}\}\!\leq\! P\Big\{\bigcup_{i=1}^{i_{\max}}\{\sup_{s\in S_i} nQ_n(s,0)\geq \widetilde{\eta}\}\Big\}\!\leq\! 
 \sum_{i=1}^{i_{\max}} P\{\sup_{s\in S_i} nQ_n(s,0)\geq \widetilde{\eta}\}.
\end{align*}

Since the sum is finite, it is enough to prove that $P\{\sup_{s\in S_i} nQ_n(s,0)\geq \widetilde{\eta}\}\to 0$  for all $i$
as $n\to\infty$. %Since we can decompose $Q_n(s,0)$ in the following manner Consider 
Consider the following decomposition
\begin{align*}
Q_n(s,0)&=\int\limits_{-\infty}^{+\infty}\bigg(\frac{1}{n^{m}}\sum_{i_1,...,i_m}\Big(g\Big(X_{i_1},...,X_{i_m},t;\lambda+\frac{s}{\sqrt{n}}\Big)-g\Big(X_{i_1},...,X_{i_m},t;\lambda\Big)
-\mu\Big(t;\lambda+\frac{s}{\sqrt{n}}\Big)+\mu(t;\lambda)\Big)\bigg)^2dM(t)
\\&=\int\limits_{-\infty}^{+\infty}\bigg(\frac{1}{n^{m}}\sum_{i_1,...,i_m}\!\Big(g\Big(X_{i_1},...,X_{i_m},t;\lambda+\frac{s}{\sqrt{n}}\Big)\!-\!
g\Big(X_{i_1},...,X_{i_m},t;\lambda+\frac{r_i}{\sqrt{n}}\Big)
-\mu\Big(t;\lambda+\frac{s}{\sqrt{n}}\Big)+\mu\Big(t;\lambda+\frac{r_i}{\sqrt{n}}\Big)
\\&+g\Big(X_{i_1},...,X_{i_m},t;\lambda+\frac{r_i}{\sqrt{n}}\Big)
-g\Big(X_{i_1},...,X_{i_m},t;\lambda\Big)-\mu\Big(t;\lambda+\frac{r_i}{\sqrt{n}}\Big)
+\mu(t;\lambda)\Big)\bigg)^2dM(t)
\\&= Q_n(s,r_i)+Q_n(r_i,0)+2\widetilde{Q_n}(s,r_i,0),
\end{align*}
where 
\begin{align}\label{mesoviti}
\begin{aligned}
 \widetilde{Q_n}(s,r_i,0)&\!=\!\int\limits_{-\infty}^{+\infty}\!\bigg(\!\frac{1}{n^{2m}}\!\sum_{i_1,...,i_m}\!\Big(g\Big(X_{i_1},...,X_{i_m},t;\lambda\!+\!\frac{s}{\sqrt{n}}\Big)\!-\!
g\Big(X_{i_1},...,X_{i_{m}},t;\lambda\!+\!\frac{r_i}{\sqrt{n}}\Big)
-\mu\Big(t;\lambda+\frac{s}{\sqrt{n}}\Big)+\mu\Big(t;\lambda+\frac{r_i}{\sqrt{n}}\Big)
\Big)\\&\times\Big(g\Big(X_{i_{m+1}},...,X_{i_{2m}},t;\lambda+\frac{r_i}{\sqrt{n}}\Big)
-g\Big(X_{i_{m+1}},...,X_{i_{2m}},t;\lambda\Big)-\mu\Big(t;\lambda+\frac{r_i}{\sqrt{n}}\Big)
+\mu(t;\lambda)\Big)\bigg)dM(t).
\end{aligned}
\end{align}
{Then the next inequality holds}
\begin{align}\label{dekompozicija}
 \sup_{s\in S_i} nQ_n(s,0)\leq \sup_{s\in S_i} nQ_n(s,r_i) + nQ_n(r_i,0)+ 2\sup_{s\in S_i} \Big|n\widetilde{Q_n}(s,r_i,0)\Big|,
\end{align}
and we shall prove that each summand is smaller than $\eta=\frac{\widetilde{\eta}}{3}$ with probability one when $n\to\infty$. %converges to zero in probability.

Define
\begin{align}\label{Hfunkcija}
\begin{aligned}
	H(x_1,...,&x_{2m};S_i)=\sup_{ s\in S_i}\int\limits_{-\infty}^{+\infty}\bigg|\Big(g\Big(x_1,...,x_m,t;\lambda+\frac{s}{\sqrt{n}}\Big)
	-\mu\Big(t;\lambda+\frac{s}{\sqrt{n}}\Big)
	-g\Big(x_1,...,x_m,t;\lambda +\frac{r_i}{\sqrt{n}}\Big)+
	\mu\Big(t;\lambda+\frac{r_i}{\sqrt{n}}\Big)\Big)\\&\times
\Big(g\Big(x_{m+1},...,x_{2m},t;\lambda+\frac{s}{\sqrt{n}}\Big)
	-\mu\Big(t;\lambda+\frac{s}{\sqrt{n}}\Big)-g\Big(x_{m+1},...,x_{2m},t;\lambda +\frac{r_i}{\sqrt{n}}\Big)
	+\mu\Big(t;\lambda+\frac{r_i}{\sqrt{n}}\Big)\Big)\bigg|dM(t)
\end{aligned}
\end{align}
Then the first summand on the right hand side of \eqref{dekompozicija} can be bounded by
\begin{align}\nonumber
 \sup_{s\in S_i} nQ_n(s,r_i)&\leq \frac{n}{n^{2m}}\sum_{i_1,\ldots,i_{2m}}H(X_{i_1},...,X_{i_{2m}};S_i)\\&
 \begin{aligned}\label{dekompSupremuma}& =\frac{n}{n^{2m}}\sum_{i_1,\ldots,i_{2m}}\big(H(X_{i_1},...,X_{i_{2m}};S_i)-EH(X_{i_1},...,X_{i_{2m}};S_i)\big)
 + \frac{n}{n^{2m}}\sum_{i_1,\ldots,i_{2m}}EH(X_{i_1},...,X_{i_{2m}};S_i)).
 \end{aligned}
\end{align}

We intend to choose $\delta$ such that the second summand of  \eqref{dekompSupremuma} is less than $\frac\eta2$
and the first one is less than $\frac\eta2$ with probability one when $n\to\infty$.

Consider the second summand of \eqref{dekompSupremuma}. Partition the set of $2m$-tuples into 
$\mathcal{I}_{c;2m}$, $c=0,1,\ldots,2m-1$, such that in $\mathcal{I}_{c;2m}$ there are $2m-c$ distinct entries. When $i_1,\ldots,i_{2m}\in\mathcal{I}_{0;2m},$ repeated application of the Cauchy-Schwarz inequality  gives 
\begin{align*}
EH&(X_1,...,X_{2m};S_i)
 \!\leq\! \int\limits_{-\infty}^{+\infty}\!E\bigg(\!\sup_{ s\in S_i}\!\bigg|\!\Big(g\Big(X_1,...,X_m,t;\lambda\!+\!\frac{s}{\sqrt{n}}\!\Big)
	\!-\!\mu\Big(t;\lambda\!+\!\frac{s}{\sqrt{n}}\!\Big)
	\!-\!g\Big(X_1,...,X_m,t;\lambda \!+\!\frac{r_i}{\sqrt{n}}\Big)\!+\!
	\mu\Big(t;\lambda\!+\!\frac{r_i}{\sqrt{n}}\Big)\!\Big)\\&\times
	\Big(g\Big(X_{m+1},...,X_{2m},t;\lambda+\frac{s}{\sqrt{n}}\Big)
	-\mu\Big(t;\lambda+\frac{s}{\sqrt{n}}\Big)-g\Big(X_{m+1},...,X_{2m},t;\lambda +\frac{r_i}{\sqrt{n}}\Big)
	+\mu\Big(t;\lambda+\frac{r_i}{\sqrt{n}}\Big)\Big)\bigg|\Bigg)dM(t)
\\&\leq \int\limits_{-\infty}^{+\infty}E\bigg(\sup_{ s\in S_i}\Big|\Big(g\Big(X_1,...,X_m,t;\lambda+\frac{s}{\sqrt{n}}\Big)
	-\mu\Big(t;\lambda+\frac{s}{\sqrt{n}}\Big)
	-g\Big(X_1,...,X_m,t;\lambda \!+\!\frac{r_i}{\sqrt{n}}\Big)\!+\!
	\mu\Big(t;\lambda+\frac{r_i}{\sqrt{n}}\Big)\Big|
	\\&\times \sup_{ s\in S_i}\Big|g\Big(X_{m+1},...,X_{2m},t;\lambda+\frac{s}{\sqrt{n}}\Big)
	-\mu\Big(t;\lambda+\frac{s}{\sqrt{n}}\Big)-g\Big(X_{m+1},...,X_{2m},t;\lambda +\frac{r_i}{\sqrt{n}}\Big)
	+\mu\Big(t;\lambda+\frac{r_i}{\sqrt{n}}\Big)\Big)\Big|\Bigg)dM(t)
	\\&= 	\int\limits_{-\infty}^{+\infty}\bigg(E\sup_{ s\in S_i}\Big|\Big(g\Big(X_1,...,X_m,t;\lambda+\frac{s}{\sqrt{n}}\Big) 	-\mu\Big(t;\lambda+\frac{s}{\sqrt{n}}\Big) 	-g\Big(X_1,...,X_m,t;\lambda +\frac{r_i}{\sqrt{n}}\Big)+	\mu\Big(t;\lambda+\frac{r_i}{\sqrt{n}}\Big)\Big|	\Bigg)^2dM(t).
\end{align*}

Applying the inequality $(a+b)^2\leq 2a^2+2b^2$, from Condition \ref{uslov3}  we get
	\begin{align*}
	EH(X_1,...,X_{2m};S_i) &\leq \!2\!\int\limits_{-\infty}^{+\infty}\!\bigg(E\bigg[\sup_{s\in S}\Big|g\Big(X_1,...,X_m,t;\lambda+\frac{s}{\sqrt{n}}\Big)
\!-\!g\Big(X_1,...,X_m,t;\lambda+\frac{r_i}{\sqrt{n}}\Big)\!\Big|\!\bigg]\!\bigg)^2dM(t)\\&+2\int\limits_{-\infty}^{+\infty}\Big(\sup_{s\in S}\Big|\mu\Big(t;\lambda+\frac{s}{\sqrt{n}}\Big)-\mu\Big(t;\lambda+\frac{r_i}{\sqrt{n}}\Big)\Big|\Big)^2 dM(t) \\&
	 \leq 2\bigg(C\frac{ \delta^2}{n}+
	 \int\limits_{-\infty}^{+\infty}{\Big(}\sup\limits_{s\in S}\Big|d_1\mu(t;\lambda)^T\frac{s-r_i}{\sqrt{n}}\Big|{\Big)^2}dM(t)\bigg)
	 \leq 2\frac{ \delta^2}{n}(C+K),
\end{align*}
where 
\begin{align}\label{constK}
 K=\int\limits_{-\infty}^{+\infty}\sum_{j=1}^p(d_1\mu(t;\lambda)_j)^2dM(t),
\end{align}
which is finite due to Condition \ref{uslov1}.

Since the cardianality of $\mathcal{I}_{0;2m}$ is $\binom{n}{2m}(2m)!$, we get 
\begin{align*}
 \frac{n}{n^{2m}}\sum_{i_1,\ldots,i_{2m}\in\mathcal{I}_{0;2m}}EH(X_1,...,X_{2m};S_i))&\leq \frac{n}{n^{2m}}\binom{n}{2m}(2m)!\cdot 2\frac{ \delta^2}{n}(C+K)
                                                                 \sim 2\delta^2(C+K), \text{ as } n\to\infty.
\end{align*}
 When $i_1,\ldots,i_{2m}\in\mathcal{I}_{c;2m},$ where $c\geq 1$, applying the same technique as before, we have
\begin{align*}
&EH(X_1,...,X_{2m};S_i)\!\leq\! \int\limits_{-\infty}^{+\infty}\!E\bigg(\!\sup_{ s\in S_i}\!\bigg|\!\Big(g\Big(X_1,...,X_m,t;\lambda\!+\!\frac{s}{\sqrt{n}}\Big)
	\!-\!\mu\Big(t;\lambda\!+\!\frac{s}{\sqrt{n}}\Big)
	\!-\!g\Big(X_1,...,X_m,t;\lambda \!+\!\frac{r_i}{\sqrt{n}}\Big)\!+\!
	\mu\Big(t;\lambda\!+\!\frac{r_i}{\sqrt{n}}\!\Big)\!\Big)\!\bigg|\\&\times\sup_{ s\in S_i}
	\bigg|\Big(g\Big(X_{m+1},...,X_{2m},t;\lambda+\frac{s}{\sqrt{n}}\Big)
	-\mu\Big(t;\lambda+\frac{s}{\sqrt{n}}\Big)-g\Big(X_{m+1},...,X_{2m},t;\lambda +\frac{r_i}{\sqrt{n}}\Big)
	+\mu\Big(t;\lambda+\frac{r_i}{\sqrt{n}}\Big)\Big)\bigg|\bigg)dM(t)
%	\\&\color{green}\leq \int\limits_{-\infty}^{+\infty} \bigg(E\bigg(\sup_{ s\in S_i}\bigg|\Big(g\Big(X_1,...,X_m,t;\lambda+\frac{s}{\sqrt{n}}\Big)
%	-\mu\Big(t;\lambda+\frac{s}{\sqrt{n}}\Big)
%	\\&-g\Big(X_1,...,X_m,t;\lambda +\frac{r_i}{\sqrt{n}}\Big)+
%	\mu\Big(t;\lambda+\frac{r_i}{\sqrt{n}}\Big)\Big)\bigg|\bigg)^2\bigg)^\frac12\bigg(E\bigg(\sup_{ s\in S_i}\bigg|\Big(g\Big(X_{m+1},...,X_{2m},t;\lambda+\frac{s}{\sqrt{n}}\Big)
%	\\&-\mu\Big(t;\lambda+\frac{s}{\sqrt{n}}\Big)
%	-g\Big(X_{m+1},...,X_{2m},t;\lambda +\frac{r_i}{\sqrt{n}}\Big)+
%	\mu\Big(t;\lambda+\frac{r_i}{\sqrt{n}}\Big)\Big)\bigg|\bigg)^2\bigg)^\frac12dM(t)
\\&\leq \!\bigg(\!\int\limits_{-\infty}^{+\infty}\! E\bigg(\!\sup_{ s\in S_i}\bigg|\Big(g\Big(X_1,...,X_m,t;\lambda+\frac{s}{\sqrt{n}}\Big)
\!-\!\mu\Big(t;\lambda+\frac{s}{\sqrt{n}}\Big)
\!-\!g\Big(X_1,...,X_m,t;\lambda \!+\!\frac{r_i}{\sqrt{n}}\!\Big)+
\mu\Big(t;\lambda+\frac{r_i}{\sqrt{n}}\Big)\Big)\bigg|\bigg)^2dM(t)\bigg)^\frac{1}{2}\\&\times\!\bigg(\!\int\limits_{-\infty}^{+\infty}\!E\bigg(\!\sup_{ s\in S_i}\!\bigg|\!\Big(g\Big(X_{m+1},...,X_{2m},t;\lambda\!+\!\frac{s}{\sqrt{n}}\!\Big)
\!-\!\mu\Big(t;\lambda\!+\!\frac{s}{\sqrt{n}}\!\Big)
\!-\!g\Big(X_{m+1},...,X_{2m},t;\lambda \!+\!\frac{r_i}{\sqrt{n}}\!\Big)\!+\!
\mu\Big(t;\lambda\!+\!\frac{r_i}{\sqrt{n}}\!\Big)\!\Big)\!\bigg|\!\bigg)^2dM(t)\!\bigg)^\frac12
\\&\leq \!\bigg(2\bigg(\int\limits_{-\infty}^{+\infty}\! E\bigg(\!\sup_{ s\in S_i}\bigg|\Big(g\Big(X_1,...,X_m,t;\lambda+\frac{s}{\sqrt{n}}\Big)
\!-\!g\Big(X_1,...,X_m,t;\lambda \!+\!\frac{r_i}{\sqrt{n}}\Big)\bigg|\!\bigg)^4dM(t)\bigg)^\frac{1}{2}\\&+2\int\limits_{-\infty}^{+\infty}\bigg(\sup_{ s\in S_i}\bigg|d_1\mu\Big(t;\lambda\Big)^T\frac{s-r_i}{\sqrt{n}}\bigg|\bigg)^2dM(t)\bigg)^\frac{1}{2}\!\bigg(\!2\bigg(\!\int\limits_{-\infty}^{+\infty}\! E\bigg(\!\sup_{ s\in S_i}\!\bigg|\!\Big(\!g\Big(\!X_{m+1},...,X_{2m},t;\lambda\!+\!\frac{s}{\sqrt{n}}\!\Big)
\!\\&-\!g\Big(\!X_{m+1},...,X_{2m},t;\lambda \!+\!\frac{r_i}{\sqrt{n}}\!\Big)\!\bigg|\!\bigg)^4\!dM(t)\!\bigg)^\frac{1}{2}+2\int\limits_{-\infty}^{+\infty}\bigg(\sup_{ s\in S_i}\bigg|d_1\mu\Big(t;\lambda\Big)^T\frac{s-r_i}{\sqrt{n}}\bigg|\bigg)^2dM(t)\bigg)^\frac12\\&
    \leq 2\Big(\varepsilon^\frac12+\frac{\delta^2}{n}K\Big).
\end{align*}

The cardinality of $\mathcal{I}_{c;2m}$ is proportional to $n^{2m-c}$ as $n\to\infty$,
hence  we get 
\begin{align*}
 \frac{n}{n^{2m}}\sum_{i_1,\ldots,i_{2m}\in\mathcal{I}_{c;2m}}EH(X_{i_1},...,X_{i_{2m}};S_i))=O\Big(\frac1{n^{c-1}}\Big).
\end{align*}
Therefore,
\begin{align*}
 \frac{n}{n^{2m}}\sum_{i_1,\ldots,i_{2m}}EH(X_{i_1},...,X_{i_{2m}};S_i))\leq 2\delta^2(C+K) + 2\sqrt{\varepsilon} + O\Big(\frac1{n}\Big).
\end{align*}

Choosing $\delta<\sqrt{\frac{\eta}{4(C+K)}}$ the second sum of \eqref{dekompSupremuma} becomes smaller than 
$\frac{\eta}2$
as $n\to\infty$. We now prove that, for this choice of $\delta$, the first summand of \eqref{dekompSupremuma}
is also smaller than $\frac{\eta}2$ with probability 1. Using the Chebyshev inequality we get
\begin{align*}
 P\Big\{&\frac{n}{n^{2m}}\sum_{i_1,\ldots,i_{2m}}
 \big(H(X_1,...,X_{2m};S_i)-EH(X_1,...,X_{2m};S_i)\big)>\frac{\eta}{2}\Big\}
 \\&\leq 
 \frac{4}{\eta^2}E\Big(\frac{n}{n^{2m}}\sum_{i_1,\ldots,i_{2m}}
 \big(H(X_1,...,X_{2m};S_i)-EH(X_1,...,X_{2m};S_i)\big)\Big)^2 
 \\&= \frac{4}{\eta^2}\frac{n^2}{n^{4m}}\sum_{i_1,\ldots,i_{4m}}{\rm Cov}\big(H(X_1,...,X_{2m};S_i),
 H(X_{2m+1},...,X_{4m};S_i)\big).
\end{align*}
Partition the set of all $4m$-tuples into sets $\mathcal{I}_{c;4m}$ defined as before. Define
\footnote[2]{Strictly speaking, $\zeta(c)$ depends not only on $c$,  however, the bounds in 
Condition \ref{uslov3} hold for all variations which makes the notation justifiable.} 
\begin{align*}
 \zeta(c)={\rm Cov}\big(H(X_1,...,X_{2m};S_i),
 H(X_{2m+1},...,X_{4m};S_i)\big), \;\; \{i_{1},\ldots,i_{4m}\}\in\mathcal{I}_{c;4m}.
\end{align*}

When $c=0$ the covariance
is equal to zero. When $c=1$, the covariance is non-zero only when one of the indices among the first $2m$ is 
equal to one among the last $2m$. In this case we have
\begin{align*}
 \zeta(1)&\leq
EH(X_1,...,X_{2m};S_i)
 H(X_{1},X_{2m+1}...,X_{4m-1};S_i)
 \\&{\leq}\int\limits_{-\infty}^{+\infty}\int\limits_{-\infty}^{+\infty}
 E\bigg[\sup_{s\in S_i}\Big|g\Big(X_1,...,X_m,t_1;\lambda+\frac{s}{\sqrt{n}}\Big)-g\Big(X_1,...,X_m,t_1;{\lambda}+\frac{r_i}{\sqrt{n}}\Big)
 -\mu\Big(t_1;\lambda+\frac{s}{\sqrt{n}}\Big)
 +\mu\Big(t_1;\lambda+\frac{r_i}{\sqrt{n}}\Big)\Big|\\&\times\sup_{s\in S_i}\Big|g\Big(X_{m+1},...,X_{2m},t_1;\lambda+\frac{s}{\sqrt{n}}\Big)
 -g\Big(X_{m+1},...,X_{2m},t_1;\lambda+\frac{r_i}{\sqrt{n}}\Big)-\mu\Big(t_1;\lambda+\frac{s}{\sqrt{n}}\Big)
 +\mu\Big(t_1;\lambda+\frac{r_i}{\sqrt{n}}\Big)\Big|
 \\&\times\!\sup_{s\in S_i}\!\Big|g\Big(\!X_1,X_{2m+1},...,X_{3m-1},t_2;\lambda\!+\!\frac{s}{\sqrt{n}}\!\Big)\!-\!g\Big(\!X_1,X_{2m+1},...,X_{3m-1},t_2;{\lambda\!+\!\frac{r_i}{\sqrt{n}}}\!\Big)\\&-\mu\Big(t_2;\lambda+\frac{s}{\sqrt{n}}\Big)
 +\mu\Big(t_2;\lambda+\frac{r_i}{\sqrt{n}}\Big)\Big|
 \cdot\sup_{s\in S_i}\Big|g\Big(X_{3m},...,X_{4m-1},t_2;\lambda+\frac{s}{\sqrt{n}}\Big)
 \\&-\!g\Big(X_{3m},...,X_{4m-1},t_2;{\lambda}\!+\!\frac{r_i}{\sqrt{n}}\Big)\!-\!\mu\Big(t_2;\lambda\!+\!\frac{s}{\sqrt{n}}\Big)
 \!+\!\mu\Big(t_2;\lambda\!+\!\frac{r_i}{\sqrt{n}}\Big)\!\Big|\bigg]dM(t_1)dM(t_2).
\end{align*}

Applying the Cauchy-Schwarz inequality and grouping back the integrals we get
\begin{align*}
&\zeta(1)\!\leq\!
 \int\limits_{-\infty}^{+\infty}\!\bigg(E\bigg[\sup_{s\in S_i}
 \Big|g\Big(X_1,...,X_m,t_1;\lambda+\frac{s}{\sqrt{n}}\Big)
 -g\Big(X_1,...,X_m,t_1;{\lambda}+\frac{r_i}{\sqrt{n}}\Big)
 -\mu\Big(t_1;\lambda+\frac{s}{\sqrt{n}}\Big)
  +\mu\Big(t_1;\lambda+\frac{r_i}{\sqrt{n}}\Big)\Big|\bigg]^2\bigg)^\frac{1}{2}\\&\times
 E\bigg[\sup_{s\in S}\Big|g\Big(X_{m+1},...,X_{2m},t_1;\lambda+\frac{s}{\sqrt{n}}\Big)
 -g(X_{m+1},...,X_{2m},t_1;{\lambda+\frac{r_i}{\sqrt{n}}})
 -\mu\Big(t_1;\lambda+\frac{s}{\sqrt{n}}\Big)
 +\mu\Big(t_1;\lambda+\frac{r_i}{\sqrt{n}}\Big)\Big|\bigg]dM(t_1)
 \\&\times \int\limits_{-\infty}^{+\infty} 
 \bigg(E\bigg[\sup_{s\in S}\Big|g\Big(X_1,X_{2m+1},...,X_{3m-1},t_2;\lambda+\frac{s}{\sqrt{n}}\Big)
  -g\Big(X_1,X_{2m+1},...,X_{3m-1},t_2;{\lambda+\frac{r_i}{\sqrt{n}}}\Big)
  -\mu\Big(t_2;\lambda+\frac{s}{\sqrt{n}}\Big)
  \\&+\mu\Big(t_2;\lambda+\frac{r_i}{\sqrt{n}}\Big)\Big|\bigg]^2\bigg)^\frac{1}{2} 
  E\bigg[\sup_{s\in S}\Big|g\Big(X_{3m},...,X_{4m-1},t_2;\lambda+\frac{s}{\sqrt{n}}\Big)
 -g\Big(X_{3m},...,X_{4m-1},t_2;{\lambda+\frac{r_i}{\sqrt{n}}}\Big)
 \\&-\mu\Big(t_2;\lambda+\frac{s}{\sqrt{n}}\Big)
 +\mu\Big(t_2;\lambda+\frac{r_i}{\sqrt{n}}\Big)\Big|\bigg]dM(t_2)
 \\&
=\bigg(\int\limits_{-\infty}^{+\infty}\bigg(E\bigg[\sup_{s\in S_i}
 \Big|g\Big(X_1,...,X_m,t_1;\lambda+\frac{s}{\sqrt{n}}\Big)
 -g\Big(X_1,...,X_m,t_1;{\lambda}+\frac{r_i}{\sqrt{n}}\Big)
 -\mu\Big(t_1;\lambda+\frac{s}{\sqrt{n}}\Big)
  +\mu\Big(t_1;\lambda+\frac{r_i}{\sqrt{n}}\Big)\Big|\bigg]^2\bigg)^\frac{1}{2} 
\\&\times\! E\bigg[\!\sup_{s\in S}\!\Big|g\Big(X_{m+1},...,X_{2m},t_1;\lambda\!+\!\frac{s}{\sqrt{n}}\!\Big)
\! -\!g\Big(X_{m+1},...,X_{2m},t_1;{\lambda\!+\!\frac{r_i}{\sqrt{n}}}\!\Big)
 \!-\!\mu\Big(t_1;\lambda\!+\!\frac{s}{\sqrt{n}}\!\Big)
 \!+\!\mu\Big(t_1;\lambda\!+\!\frac{r_i}{\sqrt{n}}\Big)\!\Big|\!\bigg]dM(t_1)\!\bigg)^2
\\&\leq \int\limits_{-\infty}^{+\infty}E\bigg[\sup_{s\in S_i}
 \Big|g\Big(X_1,...,X_m,t_1;\lambda+\frac{s}{\sqrt{n}}\Big)
 -g\Big(X_1,...,X_m,t_1;{\lambda}+\frac{r_i}{\sqrt{n}}\Big)
 \!-\!\mu\Big(t_1;\lambda\!+\!\frac{s}{\sqrt{n}}\Big)
  \!+\!\mu\Big(t_1;\lambda\!+\!\frac{r_i}{\sqrt{n}}\Big)\Big|\bigg]^2\!dM(t_1) \!
\\&\times\! \int\limits_{-\infty}^{+\infty}\! \bigg(\!E\bigg[\sup_{s\in S}\!\Big|g\Big(X_{m+1},...,X_{2m},t_1;\lambda\!+\!\frac{s}{\sqrt{n}}\!\Big)
 \!-\!g\Big(X_{m+1},...,X_{2m},t_1;{\lambda\!+\!\frac{r_i}{\sqrt{n}}}\!\Big)
 \!-\!\mu\Big(t_1;\lambda\!+\!\frac{s}{\sqrt{n}}\!\Big)
 \!+\!\mu\Big(t_1;\lambda\!+\!\frac{r_i}{\sqrt{n}}\!\Big)\!\Big|\!\bigg]\!\bigg)^2dM(t_1\!).
 \end{align*}

Applying now the inequality $(a+b)^2\leq 2a^2+2b^2$ to both factors of the product above 
we obtain
\begin{align*}
    \zeta(1)&
    %\leq\int\limits_{-\infty}^{+\infty}E\bigg[\sup_{s\in S}\Big|g\Big(X_1,...,X_m,t_1;\lambda+\frac{s}{\sqrt{n}}\Big)
    %-g\Big(X_1,...,X_m,t_1;{\lambda+\frac{r_i}{\sqrt{n}}}\Big)\!\Big|
    %\\&+\sup_{s\in S}\Big|\mu\Big(t_1;\lambda+\frac{s}{\sqrt{n}}\Big)
    %-\mu\Big(t_1;\lambda+\frac{r_i}{\sqrt{n}}\Big)\Big|\bigg]^2dM(t_1)
    %\\&\times \int\limits_{-\infty}^{+\infty} 
    %\bigg(E\bigg[\sup_{s\in S}\Big|g\Big(X_{m+1},...,X_{2m},t_1;\lambda+\frac{s}{\sqrt{n}}\Big)
    %-g\Big(X_{m+1},...,X_{2m},t_1;{\lambda+\frac{r_i}{\sqrt{n}}}\Big)\Big|\bigg]
    %\\&+\sup_{s\in S}\Big|\mu\Big(t_1;\lambda+\frac{s}{\sqrt{n}}\Big)
    %-\mu\Big(t_1;\lambda+\frac{r_i}{\sqrt{n}}\Big)\Big|\bigg)^2dM(t_1)
    \leq \Bigg(2\int\limits_{-\infty}^{+\infty}E\bigg[\sup_{s\in S}\Big|g\Big(X_1,...,X_m,t_1;\lambda+\frac{s}{\sqrt{n}}\Big)
    -g\Big(X_1,...,X_m,t_1;{\lambda+\frac{r_i}{\sqrt{n}}}\Big)\Big|\bigg]^2dM(t_1)
    \\&+2\int\limits_{-\infty}^{+\infty}
    \bigg(\sup_{s\in S}\Big|\mu\Big(t_1;\lambda+\frac{s}{\sqrt{n}}\Big)
    -\mu\Big(t_1;\lambda+\frac{r_i}{\sqrt{n}}\Big)\Big|\bigg)^2dM(t_1)\Bigg)
    \\&\times\!\Bigg(\!2\!\int\limits_{-\infty}^{+\infty}\!\bigg(\!E\bigg[\!\sup_{s\in S}\!\Big|g\!\Big(\!X_{m+1},...,X_{2m},t_1;\lambda\!+\!\frac{s}{\sqrt{n}}\!\Big)
    \!-\!g\Big(\!X_{m+1},...,X_{2m},t_1;\!{\lambda+\frac{r_i}{\sqrt{n}}}\Big)\!\Big|\!\bigg]\!\bigg)^2\!dM\!(t_1)
    \\&+2\int\limits_{-\infty}^{+\infty}\bigg(\sup_{s\in S}\Big|\mu\Big(t_1;\lambda+\frac{s}{\sqrt{n}}\Big)
    -\mu\Big(t_1;\lambda+\frac{r_i}{\sqrt{n}}\Big)\Big|\bigg)^2dM(t_1)\Bigg)
    \\&\leq  \Big(2\varepsilon^\frac{1}{2}+2\frac{\delta^2}{n}K\Big)\Big(2\frac{\delta^2}{n}(C+K)\Big)
\sim4(C+K)\varepsilon^\frac{1}{2}\frac{\delta^2}{n}+o\Big(\frac{1}{n}\Big),
\end{align*}
where $K$ is defined in \eqref{constK}.  

Let $c\geq 2$. For any such variation $X'_1,\ldots,X'_{4m}$ it holds 
\begin{align*}
 \zeta(c)&\leq \big({\rm Var} H(X'_1,...,X'_{2m};S_i)\big)^{\frac12} 
 \big({\rm Var} H(X'_{2m+1},...,X'_{4m};S_i)\big)^{\frac12}\\
        &\leq {\rm Var} H(X_{1},...,X_{2m};S_i)\leq EH^2(X_{1},...,X_{2m};S_i),
\end{align*}
 where $\{X_{1},...,X_{2m}\}$ is the variation for which the function $H$ from \eqref{Hfunkcija} has the maximal variance. 
Applying the same techniques used for the previous case we get
\begin{align*}
 \zeta(c)&\leq \int\limits_{-\infty}^{+\infty}E\bigg(\sup_{s\in S_i}
 \bigg|g\Big(X_{1},...,X_{m},t;\lambda+\frac{s}{\sqrt{n}}\Big)
 -g(X_{1},...,X_{m},t;\lambda+\frac{r_i}{\sqrt{n}})
 -\mu\Big(t;\lambda+\frac{s}{\sqrt{n}}\Big)+\mu\Big(t;\lambda+\frac{r_i}{\sqrt{n}}\Big)\bigg|\\&\times
 \sup_{s\in S}\bigg|g\Big(X_{{m+1}},...,X_{{2m}},t;\lambda+\frac{s}{\sqrt{n}}\Big)
 -g\Big(X_{{m+1}},...,X_{{2m}},t;\lambda+\frac{r_i}{\sqrt{n}}\Big)-\mu\Big(t;\lambda+\frac{s}{\sqrt{n}}\Big)
 +\mu\Big(t;\lambda+\frac{r_i}{\sqrt{n}}\Big)\bigg|\bigg)^2dM(t_1)\\&\leq
    8\bigg(\int\limits_{-\infty}^{+\infty}
    E\bigg(\sup_{s\in S_i}\bigg|g\Big(X_{1},...,X_{m},t_1;\lambda+\frac{s}{\sqrt{n}}\Big)
    -g\Big(X_{1},...,X_{m},t_1;\lambda+\frac{r_i}{\sqrt{n}}\Big)\bigg|\bigg)^4dM(t_1)
    \\&+\int\limits_{-\infty}^{+\infty}
    \bigg(\sup_{s\in S}\bigg|\mu\Big(t;\lambda+\frac{s}{\sqrt{n}}\Big)+\mu\Big(t;\lambda+\frac{r_i}{\sqrt{n}}\Big)\bigg|\bigg)^4dM(t_1)\bigg)^\frac12\\&\times\!\bigg(\!\int\limits_{-\infty}^{+\infty}\!
    E\bigg(\!\sup_{s\in S_i}\bigg|g\Big(X_{m+1},...,X_{2m},t_1;\lambda\!+\!\frac{s}{\sqrt{n}}\Big)
    \!-\!g\Big(X_{m+1},...,X_{2m},t_1;\lambda\!+\!\frac{r_i}{\sqrt{n}}\Big)\!\bigg|\!\bigg)^4\!dM(t_1)
    \\&+\int\limits_{-\infty}^{+\infty}
    \bigg(\sup_{s\in S}\bigg|\mu\Big(t;\lambda+\frac{s}{\sqrt{n}}\Big)+\mu\Big(t;\lambda+\frac{r_i}{\sqrt{n}}\Big)\bigg|\bigg)^4dM(t_1)\bigg)^\frac12
\\&<8\Big(\varepsilon+\frac{\delta^4}{n^2}K\Big).
\end{align*}

{Since the cardinality of $\mathcal{I}_c$ is proportianal to $n^{4m-c}$ (say $\kappa_c\cdot n^{4m-c}$)  we get}
\begin{align*}
 E(\sup_{s\in S_i}nQ_n(s,r_i))^2&\!\leq\! \frac{n^2}{n^{4m}}\!\sum_{i_1,\ldots,i_{4m}}\!{\rm Cov}\!\big(H(X_1,...,X_{2m};S_i),
 H(X_{2m+1},...,X_{4m};S_i)\big)\!
% \\& =\frac{n^2}{n^{4m}} \sum_{c=0}^{2m}\binom{n}{4m-c}\frac{4m!}{c!} \zeta(c)\\
 \sim\! \sum_{c=1}^{4m-1}\!  \kappa_c \cdot \frac{\zeta(c)}{n^{c-2}}   
 \!<\! {\rm const}\! \cdot \!(\varepsilon^\frac{1}{2}\delta^2 \!+\!\varepsilon) \!+\! o\Big(\frac1n\Big),
\end{align*}
and hence $P\Big\{\sup_{s\in S_i}nQ_n(s,r_i)<\frac\eta2\Big\}{\to}1$.

We now pass to the second summand of \eqref{dekompozicija}. Applying the Chebyshev inequality we get
\begin{align*}
 P\Big\{nQ_n(r_i,0)\geq \eta\Big\}&\leq \frac1{\eta^2}E(nQ_n(r_i,0))^2\\&
 =\!\frac1{\eta^2}\frac{n^2}{n^{4m}}\!\sum_{i_1,\ldots,i_{4m}}\!
 E\bigg(\!g\Big(X_{i_1},...,X_{i_m},t;\lambda\!+\!\frac{r_i}{\sqrt{n}}\Big)\!
 -\!g(X_{i_1},...,X_{i_m},t;\lambda)
 \!-\!\mu\Big(t;\lambda+\frac{r_i}{\sqrt{n}}\Big)\!\bigg)
 \\&\times\bigg(g\Big(X_{i_{m+1}},...,X_{i_{2m}},t;\lambda+\frac{r_i}{\sqrt{n}}\Big)
 -g(X_{i_{m+1}},...,X_{i_{2m}},t;\lambda)
 -\mu\Big(t;\lambda+\frac{r_i}{\sqrt{n}}\Big)\bigg)
 \\&\times\bigg(g\Big(X_{i_{2m+1}},...,X_{i_{3m}},t;\lambda+\frac{r_i}{\sqrt{n}}\Big)
 -g(X_{i_{2m+1}},...,X_{i_{3m}},t;\lambda)
 -\mu\Big(t;\lambda+\frac{r_i}{\sqrt{n}}\Big)\bigg)
 \\&\times\bigg(g\Big(X_{i_{3m+1}},...,X_{i_{4m}},t;\lambda+\frac{r_i}{\sqrt{n}}\Big)
 -g(X_{i_{3m+1}},...,X_{i_{4m}},t;\lambda)
 -\mu\Big(t;\lambda+\frac{r_i}{\sqrt{n}}\Big)\bigg).
\end{align*}
Since expectation of each factor in the sum above is zero, the expectation of the product will be zero whenever
there is at least one factor independent from all the others. Hence it is easy to see that for 
$\{i_{1},\ldots,i_{4m}\}\in \mathcal{I}_{0;4m}\cup\mathcal{I}_{1;4m}$ the summands are equal to zero.
For the rest of summands, using the previously obtained bound we get an upper bound for the general term $\mathcal{G}$
\begin{align*}
 \mathcal{G}\leq EH(X_{i_1},...,X_{i_{2m}};S)H(X_{i_{2m+1}},...,X_{i_{4m}};S)
 <8\Big(\varepsilon+\frac{d_S^4}{n^2}K\Big),
\end{align*}
where   $d_S$ is the radius od $S$. Therefore,
\begin{align*}
 E(nQ_n(r_i,0))^2<{\rm const}\cdot \varepsilon + o\Big(\frac1n\Big),
\end{align*}
and hence $nQ_n(r_i,0)$ converges to zero in probability.

For the last summand of \eqref{dekompozicija} defined in \eqref{mesoviti}, using the Cauchy-Schwarz inequality
we get
 \begin{align*}
 \sup_{s\in S_i}{|}\widetilde{Q_n}(s,r_i,0){|}&
 \leq
 \bigg(\frac{1}{n^{2m}}\sum_{i_1,...,i_{2m}}\!\int\limits_{-\infty}^{+\infty}\!{\Big(}\sup_{s\in S_i}\Big|g\Big(X_{i_1},...,X_{i_m},t;\lambda\!+\!\frac{s}{\sqrt{n}}\Big)\!-\!\mu\Big(t;\lambda\!+\!\frac{s}{\sqrt{n}}\Big)-
g\Big(X_{i_1},...,X_{i_{2m}},t;\lambda+\frac{r_i}{\sqrt{n}}\Big)
\\&+\mu\Big(t;\lambda+\frac{r_i}{\sqrt{n}}\Big)\Big|{\Big)^{2}}
dM(t)\bigg)^\frac{1}{2}
\!\bigg(
 \frac{1}{n^{2m}}\!\sum_{i_1,...,i_{2m}}\!\int\limits_{-\infty}^{+\infty}\!
 \Big(g\Big(X_{i_{m+1}},...,X_{i_{2m}},t;\lambda+\frac{r_i}{\sqrt{n}}\Big)\!\\&-\!\mu\Big(t;\lambda+\frac{r_i}{\sqrt{n}}\Big)
-g\Big(X_{i_{m+1}},...,X_{i_{2m}},t;\lambda\Big)
\Big)^2dM(t)\bigg)^\frac{1}{2}.
\end{align*}

Combining the steps used for $\sup_{s\in S_i}nQ_n(s,r_i)$ and $nQ_n(r_i,0)$, it is straightforward to show that
$2\sup_{s\in S_i}n\widetilde{Q_n}(s,r_i,0)$ converges to zero in probability, which ends the proof.

\end{proof}

\subsection{U-statistics}

We now present the analogous theorem for U-statistics with estimated parameters. Define
\begin{equation*}
    U_n(\widehat{\lambda}_n)=\frac{1}{\binom{n}{2m}}\sum_{i_1<...<i_{2m}}\Phi(X_{i_1},...,X_{i_{2m}};\widehat{\lambda}_n),
\end{equation*}
where $\Phi(x_{1},...,x_{{2m}};\widehat{\lambda}_n)$ is defined in \eqref{jezgro} and the corresponding auxiliary statistic 
\begin{equation*}
    U^*_n(\lambda)=\frac{1}{\binom{n}{2m}}\sum_{i_1<...<i_{2m}}\Phi_*(X_{i_1},...,X_{i_{2m}};\lambda),
\end{equation*}
where $\Phi_*(x_{1},...,x_{2m};\lambda)$ is defined in \eqref{Vpom}. 

\begin{theorem}\label{raspodelaUsaOcenj} 
Let $X_1,...,X_n$ be a random sample with distribution function $F$. Let all conditions of Theorem \ref{raspodela} hold, and, additionally,
\begin{equation}\label{uslovTeoreme}
   E\Big(\int\limits_{-\infty}^{+\infty}|g_1(X_{1},t;\lambda)|^2dM(t)\Big)<\infty.
\end{equation}
Then
\begin{equation*}
    nU_n(\widehat{\lambda}_n)\stackrel{D}{\rightarrow}\frac{2m(2m-1)}{2}\sum_{i=1}^\infty(\upsilon_k^*Z_k^2-\upsilon_k),
\end{equation*}
where $\{\upsilon_k^*\}$ and $\{\upsilon_k\}$, $k=1,2,...,$ are sequences of eigenvalues of operators $A^*$ and $A$, defined in \eqref{operatorA*} and \eqref{operatorA}, respectively, while $Z_{k},k=1,2,...,$ are i.i.d. standard normal random variables.
\end{theorem}

\begin{proof}
Using the result on limiting distribution of degenerate U-statistics (see e.g. \cite{korolyuk}), we get that
$$nU^*_n(\lambda)\stackrel{d}{\rightarrow}\frac{2m(2m-1)}{2}\sum\limits_{k=1}^\infty{\upsilon^*_k}(Z^2_{k}-1).$$

Thus it is enough to prove that 
\begin{align}\label{razlikaUstat}
    n(U_n(\widehat{\lambda}_n)-U_n^*(\lambda))\overset{P}\to \binom{2m}{2}\sum\limits_{k=1}^\infty(\upsilon^*_k-\upsilon_k).
\end{align}

Consider the difference 
\begin{align*}
    n(U_n(\widehat{\lambda}_n)-U_n^*(\lambda))%\\&=\!\frac{n}{n(n\!-\!1)\!\cdots\!(n\!-\!2m\!+\!1)}\!\sum\limits_{\{i_1,\ldots,i_{2m}\}\in\mathcal{I}'}\!\Big(\! \Phi(X_{i_1},...,X_{i_{2m}};\!\widehat{\lambda}_n)\!-\!\Phi_*(X_{i_1},...,X_{i_{2m}};\!{\lambda})\!\Big)
    &=\frac{n}{n(n-1)\cdots(n-2m+1)}\bigg(n^{2m}(V_n(\widehat{\lambda}_n)-V^*_n(\lambda))\\&-\sum_{{\{i_1,\ldots,i_{2m}\}\in\mathcal{I}}}\Big(\Phi(X_{i_1},...,X_{i_{2m}};\widehat{\lambda}_n)-\Phi_*(X_{i_1},...,X_{i_{2m}};{\lambda})\Big)\bigg),
\end{align*}
where the set $\mathcal{I}$ is the set of $2m$-tuples where at least one entry repeats. %defined in the proof of Theorem \ref{raspodela}. 
From the proof Theorem \ref{raspodela} we know that 
$n(V_n(\widehat{\lambda}_n)-V_n^*(\lambda))\overset{p}{\to}0$, so it remains to find the limit in probability of
\begin{align*}
    \Xi_n\!&=\!\frac{1}{(n\!-\!1)\!\cdots\!(n\!-\!2m\!+\!1)}\!\sum_{{\{i_1,\ldots,i_{2m}\}\in\mathcal{I}}}\!\Big(\!\Phi(X_{i_1},...,X_{i_{2m}};\!\widehat{\lambda}_n\!)\!-\!\Phi^*(X_{i_1},...,X_{i_{2m}};\!{\lambda})\!\Big)\\
    &=\!\frac{1}{(n\!-\!1)\!\cdots\!(n\!-\!2m\!+\!1)}\!\sum_{{\{i_1,\ldots,i_{2m}\}\in\mathcal{I}_1}}\!\Big(\!\Phi(X_{i_1},...,X_{i_{2m}};\!\widehat{\lambda}_n)\!-\!\Phi^*(X_{i_1},...,X_{i_{2m}};\!{\lambda})\!\Big)\\
    &+\frac{1}{(n\!-\!1)\!\cdots\!(n\!-\!2m\!+\!1)}\!\sum_{{\{i_1,\ldots,i_{2m}\}\in\mathcal{I}\setminus\mathcal{I}_1}}\!\Big(\!\Phi(X_{i_1},...,X_{i_{2m}};\!\widehat{\lambda}_n)\!-\!\Phi^*(X_{i_1},...,X_{i_{2m}};\!{\lambda})\!\Big)\\
    &=\Xi_n^{(1)}+\Xi_n^{(2)},
\end{align*}
where $\mathcal{I}_1$ is the set of all $2m$-tuples of indices such that only one entry repeats only once and all the others are different.  For the first summand we have
\begin{align*}
    \Xi_n^{(1)}\!&=\!
            \frac{n}{n-2m+1}\binom{2m}{2}\frac{1}{\binom{n}{2m\!-\!1}}\!\sum\limits_{i_1\neq i_2\neq\cdots\neq i_{2m-1}}\!\bigg(\Phi(X_{i_1},X_{i_1},X_{i_2},...,X_{i_{2m-1}};\!\widehat{\lambda}_n)-\Phi_*(X_{i_1},X_{i_1},X_{i_2},...,X_{i_{2m-1}};\lambda)\bigg)\\&=\frac{n}{n-2m+1}\binom{2m}{2}\frac{1}{2m-1}\frac{1}{\binom{n}{2m-1}}\sum\limits_{i_1\neq i_2\neq\cdots\neq i_{2m-1}}\bigg(\phi(X_{i_1},X_{i_2},...,X_{i_{2m-1}};\widehat{\lambda}_n)-\phi_*(X_{i_1},X_{i_2},...,X_{i_{2m-1}};\lambda)\bigg),
\end{align*}
where $\phi(X_{i_1},...,X_{i_{2m-1}};\widehat{\lambda}_n)$ and $\phi_*(X_{i_1},...,X_{i_{2m-1}};{\lambda})$ are symmetrized versions of  $\Phi(X_{i_1},X_{i_1},X_{i_2},...,X_{i_{2m-1}};\widehat{\lambda}_n)$ and $\Phi_*(X_{i_1},X_{i_1},X_{i_2},...,X_{i_{2m-1}};{\lambda})$, respectively.

Using the law of large numbers for U-statistics (see e.g. \cite{korolyuk}), and the law of large numbers for U-statistics with estimated parameters (see \cite{iverson1989effects}), we get that
\begin{align*}
       \Xi_n^{(1)} \overset{P}\to \binom{2m}{2}\frac{1}{2m-1}\big(E\phi(X_{i_1},...,X_{i_{2m-1}};\widehat{\lambda}_n)-E\phi_*(X_{i_1},X_{i_2},...,X_{i_{2m-1}};\lambda)\big).
       \end{align*}
       
    We then have
       \begin{align*}
           E\phi_*(X_{i_1},X_{i_2},...,X_{i_{2m-1}};\lambda)\big)&=
           \frac{m^2(2m-1)!}{(2m)!}E\bigg(\int\limits_{-\infty}^{+\infty}\bigg(g(X_{1},X_2,...,X_{i_m},t;\lambda)+d_1\mu(t;\lambda)^T\frac{1}{m}\sum_{j\in\{1,...,m\}}\alpha(X_{j})\bigg)\\&\times\bigg(g(X_{{1}},X_{m+1},...,X_{{2m-1}},t;\lambda)+d_1\mu(t;\lambda)^T\frac{1}{m}\sum_{k\in\{1,{m+1},..,{2m-1}\}}\alpha(X_{k})\bigg)dM(t)\bigg)
   %\\&=\color{black}
    %    \binom{2m}{2}\frac{m^2}{2m(2m-1)}\int\limits_{-\infty}^{+\infty}E\bigg(\bigg(g(X_{1},X_2,...,X_{i_m},t;\lambda)+d_1\mu(t;\lambda)^T\frac{1}{m}\sum_{j\in\{1,...,m\}}\alpha(X_{j})\bigg)
     %   \\&\color{black}\hspace{3cm}\times\bigg(g(X_{{1}},X_{m+1},...,X_{{2m-1}},t;\lambda)+d_1\mu(t;\lambda)^T\frac{1}{m}\sum_{k\in\{1,{m+1},..,{2m-1}\}}\alpha(X_{k})\bigg)\bigg)dM(t)
        \\&
  %  =\frac{m^2}{2m}\int\limits_{-\infty}^{+\infty}E\bigg(g_1(X_{1},t;\lambda)+d_1\mu(t;\lambda)^T\frac{1}{m}\alpha(X_{1})\bigg)^2dM(t)\\&
        =(2m-1)E\varphi_2(X_1,X_1; \lambda)=(2m-1)\sum\limits_{k=1}^\infty\upsilon_k^*,
    \end{align*}
    provided that the operator $A^*$ is nuclear, i.e. $\sum_{k=1}^\infty|\upsilon^*_k|<\infty$. This will be the case if $E|\varphi_2^*(X_1,X_1; \lambda)|<\infty$. Indeed,
    \begin{align*}
        E|\varphi^*_2(X_1,X_1)|&=E\bigg|\int\limits_{-\infty}^{+\infty}\bigg(g_1(X_{1},t;\lambda)+d_1\mu(t;\lambda)^T\frac{1}{m}\alpha(X_{1})\bigg)^2dM(t)\bigg|\\
        &\leq 2E\int\limits_{-\infty}^{+\infty}\bigg|g_1(X_{1},t;\lambda)\bigg|^2dM(t)+\frac{2}{m^2}E\bigg|\alpha(X_{1})\bigg|^2\cdot\int\limits_{-\infty}^{+\infty}\bigg|d_1\mu(t;\lambda)^T\bigg|^2dM(t),
    \end{align*}
   which is finite due to \eqref{uslovTeoreme} and Conditions \ref{uslov1} and \ref{uslov2}.  
   
Analogously, $E\phi(X_{i_1},...,X_{i_{2m-1}};{\widehat{\lambda}_{n}})=(2m-1)\sum_{k=1}^\infty\upsilon_k$,
and hence
\begin{align*}
    \Xi_n^{(1)} \overset{P}\to \binom{2m}{2}\sum\limits_{k=1}^\infty(\upsilon_k-\upsilon_k^*).
\end{align*}

Next we show that $\Xi_n^{(2)}\overset{P}\to0$. Let $1<r<2m$ and let $2m-r$  of $2m$ indices in the expression for $\Xi_n^{(2)}$ be different.  Then, the corresponding terms of the sum are proportional to
\begin{equation*}
        \begin{split}
            &\frac{n}{(n\!-\!2m\!+\!1)\!\cdots\!(n\!-\!2m\!+\!r)}\!\cdot\!
            \frac{1}{\binom{n}{2m\!-\!r}}\sum\limits_{i_1\neq i_2\neq\cdots\neq i_{2m-r}}\bigg(\phi^{(r)}(X_{i_1},X_{i_2},...,X_{i_{2m-r}};\widehat{\lambda}_n)-\phi^{(r)}_*(X_{i_1},X_{i_2},...,X_{i_{2m-r}};\lambda)\bigg),
        \end{split}
    \end{equation*}
        where $\phi^{(r)}(X_{i_1},X_{i_2},...,X_{i_{2m-r}};\widehat{\lambda}_n)$ and $\phi^{(r)}_*(X_{i_1},X_{i_2},...,X_{i_{2m-r}};\lambda)$ are obtained using the appropriate symmetrization of initial kernels. Using the law of large numbers for U-statistics and the Slutsky theorem, we get that the term above tends to zero in probability. This holds for all $r$, hence  $\Xi_n^{(2)}\overset{P}\to0$
        and \eqref{razlikaUstat} follows.
%\hfill $\blacksquare$
\end{proof}

\section{Application}

The main application of Theorems \ref{raspodela} and \ref{raspodelaUsaOcenj} lies in goodness-of-fit tests based on equidistribution-type characterizations. Such characterizations have the following form. Let $X_1,...,X_m$ be independent copies of random variable $X$ and let $\omega_1(\cdot)$ and $\omega_2(\cdot)$ be two functions such that
\begin{align}\label{eqChar}
\omega_1(X_1,...,X_m)\overset{d}{=}\omega_2(X_1,...,X_m) 
\end{align}
if and only if the distribution of $X$ belongs to some family $\mathcal{F}$.
The tests based on $L^2$ distance of estimators of V- or U-empirical functions of $\omega_1(\cdot)$ and $\omega_2(\cdot)$ often have the form of degenerate V- or U-statistics of order higher than 2.

One of the examples are test statistics based on V-empirical Laplace transforms from \cite{cuparic2018new},\cite{revista}. There asymptotic distributions can be obtained using Theorem \ref{raspodela}. However, the kernel in this case is differentiable, and it is possible to obtain the asymptotics directly using the mean value theorem. Here we present two examples when this is not possible and the application of our result is called for.

\subsubsection*{Example 1}
Let $X_1,\ldots,X_n$  be a random sample from distribution $F$. Consider testing the composite null  hypothesis 
$F(x)=F_0(x;\lambda)$ where $\lambda$ is the scale parameter. 
A bunch of scale families of distributions can be characterized with equidistribution-type characterizations  of the form \eqref{eqChar} where 
$\omega_1(\cdot)$ and $\omega_2(\cdot)$ are two  homogeneous functions, i.e.
\begin{align*}
    \omega_i(\lambda X_1,...,\lambda X_m)=\lambda\omega_i(X_1,...,X_m),\;\;\lambda>0.
\end{align*} 

Some examples of the characterization of this type are given in \cite{galambos1978characterizations},\cite{obradovic2015three},\cite{milosevic2016some}.

The most natural approach to construct a test is to estimate distribution functions of $\omega_1(\cdot)$ and $\omega_2(\cdot)$ respectively and to  base a test on their difference. A natural approach which yields a degenerate V-statistics is to integrate the squared difference over an empirical measure, i.e. 
\begin{align*}
%I_n&=\int(G_n^{(1)}(t)-G_n^{(2)}(t))dF_n(t)\\
%K_n&=\sup_{t}|G_n^{(1)}(t)-G_n^{(2)}(t)|\\
W_n&=\int(G_n^{(1)}(t)-G_n^{(2)}(t))^2dF_n(t),
\end{align*}
    where $F_n(t)$ is empirical distribution function. Examples of such tests can be found in \cite{IMJournal}.
    
Another possibility, currently not explored yet, is to consider 
\begin{align*}
%\widetilde{I}_n(\hat{\lambda})&=\int(\widetilde{G}_n^{(1)}(t)-\widetilde{G}_n^{(2)}(t))dM(t)\\
%\widetilde{K}_n(\hat{\lambda})&=\sup_{t}|\widetilde{G}_n^{(1)}(t)-\widetilde{G}_n^{(2)}(t)|\\
\widetilde{W}_n(\widehat{\lambda}_{n})&=\int(\widetilde{G}_n^{(1)}(t)-\widetilde{G}_n^{(2)}(t))^2dM(t),
\end{align*}
where $\widetilde{G}^{(i)}$ is U-(V-)  empirical distribution function of the  scaled sample and $\{M(t)\}$ is a finite measure. Without loss of generality we may assume that ${dM(t)}$ is a density function of some random variable.

Applying the Theorem \ref{raspodela} we get that the distribution of $n\widetilde{W}_n(\widehat{\lambda}_{n})$ coincides with  $n\widetilde{W}_n(\lambda)$ which doesn't depend on $\lambda$.  This follows from the fact that 
%{\color{purple}
%\begin{align*}
 %   \widetilde{W}_n(\hat{\lambda})&=\frac{1}{n^{4}}\sum\limits_{i_1,-..,i_{4}}\int (I\{\hat{\lambda}|X_{i_1}-X_{i_2}|<t\}-I\{\hat{\lambda}X_{i_1}<t)(I\{\hat{\lambda}|X_{i_3}-X_{i_4}|<t\}\\&-I\{\hat{\lambda}X_{i_3}<t\})e^{-t}dt)
%\end{align*}}
\begin{align*}
    \widetilde{W}_n(\widehat{\lambda}_{n})&=\frac{1}{n^{2m}}\sum\limits_{i_1,...,i_{2m}}\int (I\{\widehat{\lambda}_{n}\omega_1(X_{i_1},...,X_{i_m})<t\}-I\{\widehat{\lambda}_{n}\omega_2(X_{i_1},...,X_{i_m})\}<t)\\&\times(I\{\widehat{\lambda}_{n}\omega_1(X_{i_{m+1}},...,X_{i_{2m}})<t\}-I\{\widehat{\lambda}_{n}\omega_2(X_{i_{m+1}},...,X_{i_{2m}})<t\})dM(t),
\end{align*}
 which can be represented as
\begin{align*}
     \widetilde{W}_n(\widehat{\lambda}_{n})&=\frac{1}{n^{2m}}\sum\limits_{i_1,...,i_{2m}}\int g(X_{i_1},...,X_{i_m},t;\widehat{\lambda}_n)g(X_{i_{m+1}},...,X_{i_{2m}},t;\widehat{\lambda}_n)dM(t),
\end{align*}
where
\begin{align*}
    g(x_1,...,x_m,t;\lambda)\!=\!\frac{1}{m!}\!\sum_{\pi(m)}\! (I\!\{\!\lambda\omega_1(x_{\pi(1)},...,x_{\pi(m)})\!<\!t\}\!-\!I\!\{\!\lambda\omega_2(x_{\pi(1)},...,x_{\pi(m)})\}\!<\!t).
\end{align*}

Under $H_0$ it holds
%{\color{purple}
%\begin{align*}
 %   \mu(t;\gamma)&=E_\lambda(g(X_1,X_2,t;\gamma))\\&=E_\lambda(I\{\gamma\|X_1-X_2|<t\}-I\{\gamma\omega_2(X_1,...,X_m)<t\})\\&=P_\lambda\{\gamma\omega_1(X_1,...,X_m)<t\})-P_\lambda\{\gamma\omega_2(X_1,...,X_m)<t\})=0
%\end{align*}
 % }
\begin{align*}
    \mu(t;\gamma)%&=E_\lambda(g(X_1,..,X_m,t;\gamma))\\&%=E_\lambda(I\{\gamma\omega_1(X_1,...,X_m)<t\}-I\{\gamma\omega_2(X_1,...,X_m)<t\})\\&
    =P_\lambda\{\gamma\omega_1(X_1,...,X_m)<t\})-P_\lambda\{\gamma\omega_2(X_1,...,X_m)<t\})=0.
\end{align*}
  
Last equality holds for each $\gamma$ due to the characterization. Therefore the first derivative  $d_1\mu(t;\lambda)$ will be also equal to zero.

 Consider now the particular case of the Puri-Rubin characterization \cite{puri1970characterization}, i.e. when $\omega_1(X_1,X_2)=|X_1-X_2|$ and $\omega_2(X_1)=X_1$, which implies that $F$ is exponential distribution with some scale parameter $\lambda$, and let $dM(t)=e^{-t}dt$.

Then the corresponding test statistic is
\begin{align*}
    \widetilde{W}^{\textup{PR}}_n(\widehat{\lambda}_{n})&=\frac{1}{n^{4}}\sum\limits_{i_1,...,i_{4}}\int (I\{\widehat{\lambda}_{n}|X_{i_1}-X_{i_2}|<t\}-I\{\widehat{\lambda}_{n}X_{i_1}<t\})(I\{\widehat{\lambda}_{n}|X_{i_3}-X_{i_4}|<t\}-I\{\widehat{\lambda}_{n}X_{i_3}<t\})e^{-t}dt,
\end{align*}
where $\widehat{\lambda}_{n}=\bar{X}^{-1}$.

According to the argument above, it is enough to obtain the asymptotic distribution of $n\widetilde{W}^{\textup{PR}}_n(\lambda)$. The symmetrized kernel of $\widetilde{W}^{\textup{PR}}_n(\lambda)$ is
\begin{align*}
   \Phi^{\textup{PR}}_n(x_1,x_2,x_3,x_4;\lambda)&=\frac{1}{4!}\sum\limits_{\pi\in\Pi(4)}  \bigg(e^{-\lambda\max (\left| x_{\pi(1)}-x_{\pi(2)}\right| ,\left| x_{\pi(3)}-x_{\pi(4)}\right| )}- e^{-\lambda \max (x_{\pi(3)},\left| x_{\pi(1)}-x_{\pi(2)}\right| )}- e^{-\lambda\max (x_{\pi(1)},\left| x_{\pi(3)}-x_{\pi(4)}\right| )}\\&+e^{-\lambda\max (x_{\pi(1)},x_{\pi(3)})}\bigg).
\end{align*}
 Since $\widetilde{W}^{\textup{PR}}_n(\widehat{\lambda}_n)$ is scale-free, we may assume that $\lambda=1$. The second projection is then equal to
\begin{equation*}
     \begin{split}
      \varphi^\textup{PR}_2(s,t)&=\frac{1}{18}+\frac{1}{2}(e^{-2s-t}+e^{-s-2t})-\frac{1}{4}(e^{-2t}+e^{-2s})-\frac{16}{9}e^{-s-t}+\frac{1}{9}e^{-\min(s,t)}(2-3\min(s,t))\\&+\frac{1}{18}e^{-\max(s,t)}(19-6\min(s,t)).
     %\varphi^\textup{PR}_2(s,t)&=\frac{1}{30}+\frac{3}{10}(e^{-2s-t}+e^{-s-2t})-\frac{3}{20}(e^{-2t}+e^{-2s})-\frac{16}{15}e^{-s-t}\\&+\frac{1}{15}e^{-\min(s,t)}(2-3\min(s,t))+\frac{1}{30}e^{-\max(s,t)}(19-6\min(s,t)),
     \end{split}
\end{equation*}

Conditions \ref{uslov1} and \ref{uslov2} obviously hold and Condition \ref{uslov3} follows from the fact that the kernel is a linear combination of indicators.

Hence, the asymptotic distribution follows from \eqref{raspodelaPoznato} with the corresponding integral operator being 
\begin{equation*}
    A^{\textup{PR}}q(x)=\int_{R}\varphi^{\textup{PR}}_2(x,y)q(y)e^{-y}dy.
\end{equation*}
The eigenvalues of $A^{\textup{PR}}$ cannot be obtained analytically, however they can be approximated numerically using the method from \cite{bozin2020new}.

\subsubsection*{Example 2}
The second example comes from testing hypothesis within the location family. In testing goodness-of-fit  based on equidistribution characterizations, it is often the case that estimating a location parameter, unlike the scale one,  changes the asymptotic distribution (see e.g. \cite{milosevic2019comparison}).

Let $X_1,\ldots,X_n$  be a random sample from distribution $F$. Consider testing the composite null normality hypothesis 
$F(x)=\Phi(\frac{x-\theta}{\sigma})$, where both $\theta$ and $\sigma$ are unknown, based on the famous Polya's characterization
\cite{polya1923herleitung}, arguably the first ever published equidistribution-type characterization. It states that if $X_1$ and $X_2$ are i.i.d. random variables with distribution function $F$, then the equidistribution 
\begin{align}\label{poljaKara}
    \frac{X_1+X_2}{\sqrt{2}}\overset{d}{=}X_1
\end{align}
implies that $F$ is normal with zero mean and arbitrary variance.
Some normality tests based on this characterization can be found in \cite{muliere2002scale} and \cite{litvinova2006two}. 

Let $F_n(t)$ be the usual empirical distribution function and let
\begin{align*}
 G_n(t)=\frac1{n^2}\sum_{i,j=1}^n {\rm I}\Big\{\frac{X_i+X_j}{\sqrt{2}}\leq t\Big\}
\end{align*}
be the V-empirical distribution function associated with the random variable in \eqref{poljaKara}.

Here we consider an $\omega^2$-type test
statistic
\begin{align}
 T_n(\widehat{\theta})=\int_{-\infty}^\infty (\widetilde{G}_n(t)-\widetilde{F}_n(t))^2 d\widetilde{F}_n(t),
\end{align}
where 
$\widetilde{F}_n(t)$ and $\widetilde{G}_n(t)$ are the aforementioned empirical d.f.'s applied to the shifted sample $X^*_i=X_i-\widehat{\theta}$, and $\widehat{\theta}=\bar{X}$,
the sample mean. It is easy to see that the statistic is location and scale free. After transformation we obtain
\begin{align*}
    T_n(\widehat{\theta})=
    \int_{-\infty}^\infty(G_n(s-\widehat{\theta}+\widehat{\theta}\sqrt{2})-F_n(s))^2dF_n(s).
\end{align*}

However, in order to apply Theorem \ref{raspodela} we need to get rid of the empirical measure.
Define
\begin{align*}
    \overline{T}_n(\widehat{\theta})=
    \int_{-\infty}^\infty(G_n(s-\widehat{\theta}+\widehat{\theta}\sqrt{2})-F_n(s))^2dF(s),
\end{align*}
where $F(s)=\Phi(\frac{s-\theta}{\sigma})$, where $\theta$ and $\sigma$ are true parameter values. Since the test statistic is location-scale invariant, we assume $\theta=0$ and $\sigma=1$.

%{\color{purple} Ovo je ok kad se bas bas udubis jer ce se raspodela od $\bar{T}_n(\hat{\theta})$ rachunati kad su stvarne vrednosti $\theta$ i $\sigma$ ali nije jasno da  u samoj konstrukciji jer ce to biti izraz koji zavisi od $\theta$i $\sigma$}

Put $W_n(s)=(G_n(s-\widehat{\theta}+\widehat{\theta}\sqrt{2})-F_n(s))$. Then, using the law of large numbers for 
V-statistics, the fact that $\sqrt{n}W_n(s)$ converges to a centered Gaussian
process, continuous mapping theorem, and Donsker theorem, we get
\begin{align*}
 nT_n(\widehat{\theta})-n\overline{T}_n(\widehat{\theta})=
 \int\frac{1}{\sqrt{n}} nW^2_n(s)d(\sqrt{n}(F_n(s)-F(s))\stackrel{d}{\to} \int 0 d\mathbb{G}_F=0,
\end{align*}
where $\mathbb{G}_F=B\circ F$ and  $\{B(s)\}$ the standard Brownian bridge. Hence the statistics are asymptotically equivalent.

%Using the Hadamard derivative and the functional delta method it follows that 
%\begin{align*}
% nT_n(\hat{\theta})-n\bar{T}_n(\hat{\theta})\overset{P}{\to} 0,
%\end{align*}
%and hence the statistics are asymptotically equivalent.

The statistic  $\bar{T}_n(\widehat{\theta})$ is a V-statistic with kernel of the form \eqref{jezgro} 
\begin{align*}
        \Psi(x_{1},...,x_{{4}};\widehat{\lambda}_n)=\frac{1}{4!}\sum\limits_{\pi\in\Pi(4)}\int\limits_{-\infty}^{+\infty} 
        g(x_{\pi(1)},x_{\pi(2)},t;\widehat{\theta}_n)g(x_{\pi(3)},x_{\pi(4)},t;\widehat{\theta}_n)dM(t),
\end{align*}
where
\begin{align*}
    g(x_1,x_2,s;\gamma)&=I\bigg\{\frac{x_1+x_2}{\sqrt{2}}\leq s+\gamma(\sqrt{2}-1)\bigg\}-\frac{1}{2}I\{x_1\leq s\}-\frac{1}{2}I\{x_2\leq s\}.
\end{align*}
We now show that the conditions for applying Theorem \ref{raspodela} are fulfilled. Taking into account that $\theta=0$,
\begin{align*}
    \mu(s;\theta)=E_{\theta}( g(x_1,x_2,s;\gamma))|_{\gamma=\theta}=(\Phi(s+\gamma(\sqrt{2}-1))-\Phi(s))|_{\gamma=\theta}=0.
\end{align*}
In addition $d_1\mu(s;\gamma)=\phi(s+\gamma(\sqrt{2}-1))\cdot(\sqrt{2}-1),$
%\begin{align*}
%    d_1\mu(s;\gamma)=\phi(s+\gamma(\sqrt{2}-1))\cdot(\sqrt{2}-1),
%\end{align*}
where $\phi(x)$ is the standard normal density.
Therefore,
%\begin{equation*}
%    \int\limits_{-\infty}^{+\infty}(\phi(s)\cdot(\sqrt{2}-1))^2\phi(s)ds<\infty, \text{ and}
%\end{equation*}
%and
\begin{equation*}
\begin{split}
&\int\limits_{-\infty}^{+\infty}(\phi(s)\cdot(\sqrt{2}-1))^2\phi(s)ds<\infty, \\
    &\frac{1}{\gamma^2}\int\limits_{-\infty}^{+\infty}(\Phi(s+\gamma(\sqrt{2}-1))-\Phi(s)-\phi(s)\cdot(\sqrt{2}-1)\cdot\gamma)^2\phi(s)ds%=\frac{1}{\gamma^2}\int\limits_{-\infty}^{+\infty}\left(\frac{\gamma^2(\sqrt{2}-1)^2}{2}\phi'(s+\xi(s))\right)^2\phi(s)ds\\&
    ={\gamma^2}\frac{(\sqrt{2}-1)^4}{4}\int\limits_{-\infty}^{+\infty}\left(\phi'(s+\xi(s))\right)^2\phi(s)ds<\varepsilon,
    \end{split}
\end{equation*}
where the last inequality follows from the boundness of the function $\phi'(s)$. Hence, Condition \ref{uslov1} is satisfied.
Condition \ref{uslov2} holds obviously for the sample mean.
Condition \ref{uslov3} is straightforward to verify using the properties of the normal density and the finiteness of the first and second moment of the kernel $\Psi_*$.
The  second projection is equal to
\begin{align*}
    \psi^*_2(x_1,x_{3};\theta)%=\frac{16}{24} E(\Psi_*(x_1,X_2,x_3,X_{4};\theta))\\&%=   \frac{2}{3}\int\limits_{-\infty}^{+\infty}\bigg(g_1(x_1,s;\theta)+d_1\mu(s;\theta)\frac{x_1}{m}\bigg)   \bigg(g_1(x_{3},s;\theta)+d_1\mu(s;\theta)\frac{x_3}{m}\bigg)\phi(s) ds\\&
    &=
   \frac{2}{3}\int\limits_{-\infty}^{+\infty}\bigg(\Phi(\sqrt{2}s-x_1)-\frac{1}{2}\Phi(s)-\frac{1}{2}I\{x_1<s\}+\phi(s)\cdot(\sqrt{2}-1)\frac{x_{1}}{m}\bigg)\\&\times
   \bigg(\Phi(\sqrt{2}s-x_3)-\frac{1}{2}\Phi(s)-\frac{1}{2}I\{x_3<s\}+\phi(s)\cdot(\sqrt{2}-1)\frac{x_{3}}{m}\bigg) \phi(s)ds.%\\&=   \frac{2}{3}\int\limits_{-\infty}^{+\infty}\Phi(\sqrt{2}s-x_1)\Phi(\sqrt{2}s-x_3)\phi(s)ds-\frac{1}{3}\int\limits_{-\infty}^{+\infty}\Phi(\sqrt{2}s-x_1)\Phi(s)\phi(s)ds-\frac{1}{3}\int\limits_{-\infty}^{+\infty}\Phi(\sqrt{2}s-x_3)\Phi(s)\phi(s)ds\\&-\frac{1}{3}\int\limits_{-\infty}^{+\infty}\Phi(\sqrt{2}s-x_1)I\{x_3<s\}\phi(s)ds+\frac{2}{3}(\sqrt{2}-1)\frac{x_{3}}{m}\int\limits_{-\infty}^{+\infty}\Phi(\sqrt{2}s-x_1)\phi^2(s)ds+...
\end{align*}
The expression above can be calculated and expressed as a function of bivariate normal distributions using the formulae from \cite{owen1980table}.
The asymptotic distribution of $nT_n(\widehat{\theta})$ now follows from Theorem \ref{raspodela} and the corresponding eigenvalues can be obtained numerically using the method from \cite{bozin2020new}.

 \section*{Acknowledgement}
The authors express their deep gratitude to two anonymous referees whose suggestions improved the quality of the paper.

\bibliographystyle{plain}
\bibliography{literatura.bib}

\end{document}